\documentclass[10pt]{article}
\usepackage{latexsym}
\usepackage{amssymb}
\usepackage{amsmath}
\usepackage{mathrsfs}
\usepackage{graphicx}
\usepackage{pstricks} 
\usepackage{verbatim}
\usepackage[all]{xy}
\usepackage{hyperref}
\usepackage{longtable}
\usepackage{multirow}

\newtheorem{theo}{Theorem}[section]
\newenvironment{example}{ \smallskip \addtocounter{theo}{1} \noindent \textbf{Example  \arabic{section}.\arabic{theo}~}}{}
\newtheorem{lemma}[theo]{Lemma}

\newtheorem{corollary}[theo]{Corollary}
\newtheorem{prop}[theo]{Proposition}

\newenvironment{proof}[1][]{ \textbf{Proof#1. }}{$\Box$\medskip}
\newtheorem{defi}[theo]{Definition}

\newcommand{\WW}{\mathbf{W}}
\newcommand {\ZZ} {\mathbb {Z}}
\newcommand {\QQ} {\mathbb {Q}}

\newcommand {\CC} {\mathbb {C}}

\newcommand{\sS}{\mathfrak{s}}
\newcommand{\tT}{\mathfrak{t}}

\newcommand{\kk}{\mathfrak{k}}

\newcommand{\nn}{\mathfrak{n}}

\newcommand{\hh}{\mathfrak{h}}

\newcommand{\qq}{\mathfrak{q}}
\newcommand{\pp}{\mathfrak{p}}
\newcommand{\gG}{\mathfrak{g}}
\newcommand{\bb}{\mathfrak{b}}

\newcommand{\lL}{\mathfrak{l}}

\renewcommand {\phi} {\varphi}

\newcommand{\half}{\frac{1}{2}}

\newcommand{\refth}[1]{Theorem \ref{#1}}
\newcommand{\refle}[1]{Lemma \ref{#1}}
\newcommand{\refsec}[1]{section \ref{#1}}
\newcommand{\refcor}[1]{Corollary \ref{#1}}
\newcommand{\refprop}[1]{Proposition \ref{#1}}

\newcommand{\refeq}[1]{(\ref{#1})}
\newcommand{\refdef}[1]{Definition \ref{#1}}

\newcommand{\sL}{\mathrm{sl}}

\newcommand{\sP}{\mathrm{sp}}
\newcommand{\linspan}{\mathrm{span}}

\newcommand{\so}{\mathrm{so}}

\newcommand{\Cone}{\mathrm{Cone}}

\newcommand{\hw}{\mathrm{Sing}}

\newcommand{\eps}{\varepsilon}

\newcommand{\rk}{\mathrm{rk}}
\def\cplus{\hbox{$\subset${\raise0.3ex\hbox{\kern -0.55em ${\scriptscriptstyle +}$}}\ }}

\def\clplus{\hbox{$\subset${\raise0.3ex\hbox{\kern -0.55em ${\scriptscriptstyle +}$}}\ }}
\def\crplus{\hbox{$\supset${\raise1.05pt\hbox{\kern -0.60em ${\scriptscriptstyle +}$}}\ }}
\def\nperp{\hbox{$\perp${\raise1.05pt\hbox{\kern -0.55em $/$}}\ }}
\def\sperp{\hbox{$\perp${\hbox{\kern -.77em $-$}\raise-1pt\hbox{\kern -.77em $-$}}\ }}
\def\nsimeq{\hbox{$\simeq${\raise1.05pt\hbox{\kern -0.55em $/$}}\ }}
\newcommand{\pma}{\pm}

\newcommand{\pmf}[1]{ \pm}
\newcommand{\arxivVersion}[1]{#1}
\author{Todor Milev}
\title{
Root Fernando-Kac subalgebras of finite type
}

\begin{document}
\maketitle
\begin{abstract}
Let $\gG$ be a finite-dimensional Lie algebra and $M$ be a $\gG$-module. The Fernando-Kac subalgebra of $\gG$ associated to $M$ is the subset $\gG[M]\subset\gG$ of all elements $g\in\gG$ which act locally finitely on $M$. A subalgebra $\lL\subset\gG$ for which there exists an irreducible module $M$ with $\gG[M]=\lL$ is called a Fernando-Kac subalgebra of $\gG$. A Fernando-Kac subalgebra of $\gG$ is of finite type if in addition $M$ can be chosen to have finite Jordan-H\"older $\lL$-multiplicities. Under the assumption that $\gG$ is simple, I. Penkov has conjectured an explicit combinatorial criterion describing all Fernando-Kac subalgebras of finite type which contain a Cartan subalgebra. In the present paper we prove this conjecture for $\gG\nsimeq E_8$. \end{abstract}

\section{Introduction} 

Let $\gG$ be a finite-dimensional complex Lie algebra and $M$ be a $\gG$-module. \emph{The Fernando-Kac subalgebra} $\gG[M]\subset\gG$ associated to $M$ is by definition the subset of elements of $\gG$ which act locally finitely on $M$. The fact that $\gG[M]$ is a subalgebra of $\gG$ was independently proved by V. Kac and S. Fernando, \cite{Kac:FernandoKacSubalgebra}, \cite{Fernando1}. The $\gG$-module $M$ is said to be a \emph{$(\gG,\kk)$-module} if $\kk\subset\gG[M]$, and to be a \emph{strict} $(\gG,\kk)$-module if $\gG[M]=\kk$. A subalgebra $\kk$ of $\gG$ is defined to be a \emph{Fernando-Kac subalgebra} if there exists an irreducible strict $(\gG,\kk)$-module $M$. Under a \emph{root subalgebra} of $\gG$ we understand a subalgebra of $\gG$ which contains a Cartan subalgebra $\hh\subset\gG$. 

A $(\gG,\kk)$-module $M$ is of \emph{finite type} if for any fixed irreducible finite-dimensional $\kk$-module $V$ the Jordan-H\"older multiplicities of $V$ in all finite-dimensional $\kk$-submodules of $M$ are uniformly bounded. A Fernando-Kac subalgebra $\kk$ of $\gG$ is of \emph{finite type} if $\kk=\gG[M]$ for some irreducible $(\gG,\kk)$-module $M$ of finite type. Otherwise, $\kk$ is of \emph{infinite type}. In what follows, $\gG$ will be assumed reductive. 

This paper completes the classification of the root Fernando-Kac subalgebras of finite type of the classical simple Lie algebras. This classification was initiated in \cite{PS:GenHarishChandra,PSZ}. It is proved in \cite{PS:GenHarishChandra} that every root subalgebra of $\gG$ is a Fernando-Kac subalgebra, not necessarily of finite type. In \cite{PSZ}, I. Penkov, V. Serganova and G. Zuckerman gave a construction of an infinite family of irreducible $(\gG,\kk)$-modules of finite type by using a procedure of geometric induction from irreducible $(\pp_{red},\hh)$-modules of finite type ($\pp_{red}$ stands for the reductive part of a parabolic subalgebra $\pp$). This enabled them to determine all root Fernando-Kac subalgebras of finite type for $\gG=\sL(n)$, as they showed that all such subalgebras arise through this construction. Furthermore, I. Penkov conjectured an explicit description of all root Fernando-Kac subalgebras of finite type in terms of their root systems. It claims that two conditions, the ``cone condition'' and the ``centralizer condition'' (see \refdef{defCone} below), are necessary and sufficient for a subalgebra to be a root Fernando-Kac of finite type (\refth{thMainResult}). The ``centralizer condition'' is a consequence of S. Fernando's result \cite{Fernando1} that Lie algebras of type $B$ and $D$ do not admit strict irreducible $(\gG,\hh)$-module of finite type, and is trivially satisfied in both types $A$ and $C$. Moreover, it can be verified that Penkov's conjecture is compatible with decomposing a semisimple subalgebra into simple ideals, hence it suffices to prove it for the simple Lie algebras.

In the present paper, we prove Penkov's conjecture for all simple Lie algebras except $E_8$. The proof of \refth{thMainResult} has two distinct parts. The first part establishes that all root subalgebras $\lL$ which do not satisfy the cone condition are Fernando-Kac subalgebras of infinite type. This is the main contribution of the present paper and makes up for all of sections \ref{secConeIntersections} and \ref{secBigTedious}. In section \ref{secConeIntersections}, for any root subalgebra $\lL$ we give a combinatorial definition of an $\lL$-infinite weight (Definition \ref{defStronglyPerp}), equivalent to the existence of certain $\sL(2)$- subalgebras of $\gG$ with pairwise strongly orthogonal roots. Then, assuming the existence of an $\lL$-infinite weight, we construct a $\kk$-type of infinite multiplicity in any strict $(\gG,\lL)$-module. 

In section \ref{secBigTedious}, we complete this first part of the proof by showing that, for a simple Lie algebra $\gG\nsimeq E_8$, the failure of the cone condition implies the existence of an $\lL$-infinite weight. The argument for the classical Lie algebras goes through classifying minimal cone intersection relations (Lemma \ref{leBigTedious}). The proof for the exceptional Lie algebras $F_4$, $E_6$ and $E_7$ involves a computer computation, performed by a C++ program written by the author\footnote{The program and its source code are publicly available under the name ``vector partition'' program}. It is our conjecture that the failure of the cone condition implies the existence of an $\lL$-infinite weight for the exceptional Lie algebra $E_8$ as well; we have not yet been able to prove (or disprove) this fact due to the size of the computation. The latter issue should be resolved algorithmically in the near future. 

To complete the proof, for a given root subalgebra $\lL=\kk\crplus\nn$ satisfying the cone and the centralizer conditions, one constructs a strict irreducible $(\gG,\lL)$-module $M$. Such a general construction is already contained in \cite{PSZ}; here, we only show that any root subalgebra $\lL$ as above provides input data for it. This is ensured by \refprop{lePara}.

\textbf{Acknowledgements.} I would like to thank my advisor Ivan Penkov for guiding my study of Fernando-Kac subalgebras of finite type and for pointing out a number of inaccuracies in preliminary drafts. I would also like to thank Pavel Tumarkin for his very helpful criticism of the text. Finally, I acknowledge the funding of my studies through the DAAD (Deutscher Akademischer Austauschdienst).

\arxivVersion{\tableofcontents}

\section{Preliminaries}
The base field is $\CC$. $U(\bullet)$ stands for universal enveloping algebra. A reductive Lie algebra $\gG$ and its Cartan subalgebra $\hh$ are assumed fixed. All root subalgebras of $\gG$ we consider are assumed to contain $\hh$. The $\hh$-roots of $\gG$ are denoted by $\Delta(\gG)$. By $\lL$ we denote a variable root subalgebra of $\gG$ with nilradical $\nn$. We denote by $\kk$ the unique reductive part of $\lL$ which contains $\hh$, and we write $\lL={ \kk}\crplus\nn$. The root spaces of $\lL$ are automatically root spaces of $\gG$, and we denote by $\Delta(\lL)$ (respectively, $\Delta(\kk)$) the set of roots of $\lL$ (respectively of $\kk$); $\Delta(\kk)\subset\Delta(\lL)\subset\Delta(\gG)$. We also put $\Delta(\nn):=\Delta(\lL)\backslash \Delta(\kk)$. There are vector space decompositions 
\[
\begin{array}{cc}
\displaystyle\gG=\hh\oplus\bigoplus_{\alpha\in\Delta(\gG)}\gG^\alpha,& \displaystyle\lL=\hh\oplus\bigoplus_{\alpha \in \Delta(\lL)} \gG^{\alpha},\\
\displaystyle \kk=\hh\oplus\bigoplus_{\substack{ \alpha\in\Delta(\lL) :\\-\alpha \in\Delta(\lL)}}\gG^\alpha,&\displaystyle\nn=\bigoplus_{ \substack{\alpha\in \Delta(\lL) :\\ -\alpha \notin \Delta(\lL)}}\gG^{\alpha}\\
\end{array}.
\] 
We fix a Borel subalgebra $\bb\supset\hh$ whose roots are by definition the \emph{positive} roots; we denote them by $\Delta^+(\gG)$. Given a set of roots $I$, we denote by $\Cone_{\ZZ}(I)$ (respectively, $\Cone_{\QQ}(I)$) the $\ZZ_{\geq 0}$-span (respectively, $\QQ_{\geq 0}$-span) of $I$.

The form on $\hh^*$ induced by the Killing form is denoted by $\langle\bullet,\bullet\rangle$. The sign $\sperp$ stands for \emph{strongly orthogonal}; two roots $\alpha,\beta$ are defined to be strongly orthogonal if neither $\alpha+\beta$ nor $\alpha-\beta$ is a root or zero (which implies $\langle\alpha,\beta\rangle=0$).  We say that a root $\alpha$ is \emph{linked} to an arbitrary set of roots $I$ if there is an element of $I$ that is not orthogonal to $\alpha$. The Weyl group of $\gG$ is denoted by $\WW$. For two roots $\alpha,\beta\in\Delta(\gG)$ we say that $\alpha\preceq\beta$ if $\beta-\alpha$ is a non-negative linear combination of positive roots.

For any subalgebra $\sS\subset\gG$ we denote by $N(\sS)$ (respectively, $C(\sS)$) the normalizer (respectively, the centralizer) of $\sS$ in $\gG$. If $\sS$ is reductive, we set $\sS_{ss}=[\sS,\sS]$, and if $\sS$ contains $\hh$ we denote by $\sS_{red}$ the reductive part of $\sS$ containing $\hh$ (in particular, $\kk=\lL_{red}$). 

\begin{lemma}\label{leC(kss)}
Let $\kk\subset\gG$ be a reductive root subalgebra. Then $C(\kk_{ss})= \hh_1\oplus$ $\bigoplus_{ \alpha\sperp\Delta(\kk)}$ $\gG^{\alpha}$, where $\hh_1=\{h\in\hh~|~\gamma(h)=0, \forall \gamma \in \Delta( \kk)\}$. 
\end{lemma}
\begin{proof}
Let $x:=h+\sum_{\alpha\in\Delta(\gG)}a_\alpha g^{\alpha}\in C(\kk_{ss})$, where $g^\alpha\in\gG^\alpha$. For any $\gamma\in\Delta(\kk)$ we have $0=[x,g^\gamma] =\gamma(h)g^\gamma+ \sum_{\alpha\in\Delta(\gG)} a_\alpha c_{\alpha \gamma} g^{\alpha +\gamma}$, where $g^{\alpha +\gamma} =0$ if $\alpha+\gamma$ is not a root and $c_{\alpha\gamma}\neq 0$ whenever $\alpha+\gamma$ is a root. Therefore $a_\alpha=0$ for all $\alpha$ which are not strongly orthogonal to $\Delta(\kk)$ and $\gamma(h)=0$. On the other hand, it is clear that when $\gamma(h)=0$ for all $\gamma\in \Delta(\kk)$, and $a_\alpha$ are arbitrary, then $h+\sum_{\alpha\sperp\Delta(\kk)}a_\alpha g^\alpha$ is an element of $C(\kk_{ss})$.
\end{proof}

We fix the conventional expressions for the positive roots of the classical root systems:
\begin{eqnarray*}
A_n, n\geq 2&:& \Delta^+(\gG)=\{\eps_i-\eps_j|i< j\in\{1,\dots,n+1\}\};\\
B_n, n\geq 2&:& \Delta^+(\gG)=\{\eps_i\pma\eps_j|i< j\in\{1,\dots,n\}\}\cup\{\eps_i|i\in\{1,\dots n\}\};\\
C_n, n\geq 2&:& \Delta^+(\gG)=\{\eps_i\pma\eps_j|i\leq j\in\{1,\dots,n\}\}\backslash\{0\};\\
D_n, n\geq 4&:& \Delta^+(\gG)=\{\eps_i\pma\eps_j|i< j\in\{1,\dots,n\}\}.
\end{eqnarray*}

\section{Statement of the result}
\begin{defi}\label{defCone}
(I. Penkov) 
\begin{itemize}
\item[(a)] \emph{Cone condition}. We say that $\lL$ satisfies the cone condition if $\Cone_{\QQ}(\Delta(\nn))$ $\cap$ $ \Cone_{\QQ}( \hw_{ \bb \cap \kk} (\gG/\lL))$ $=\{0\}$, where $\hw_{\bb\cap\kk}(\gG/\lL):= \{\alpha \in \Delta (\gG)\backslash\Delta(\lL) | \alpha + \delta \notin \Delta(\gG), \forall \delta \in\Delta(\bb)\cap\Delta(\kk)\}$ are the weights of the $\bb\cap\kk$-singular vectors of the $\kk$-module $\gG/\lL$.
\item[(b)] \emph{Centralizer condition}. We say that $\lL$ satisfies the centralizer condition if a (equivalently any) Levi subalgebra of the Lie algebra $C(\kk_{ss})\cap N(\nn)$ has simple constituents of type $A$ and $C$ only.
\end{itemize}
\end{defi}
\textbf{Remark.} The cone condition (a) holds as stated if and only if it holds with $\QQ$ replaced by $\ZZ$. 

\begin{theo}\label{thMainResult}

Let $\lL=\kk\crplus\nn$ be a root subalgebra  of $\gG\simeq\sL(n)$, $\so(2n+1)$, $\sP(2n)$, $\so(2n)$, $G_2$, $F_4$, $E_6$ or $E_7$. Then $\lL$ is a Fernando-Kac subalgebra of finite type if and only if $\lL$ satisfies the cone and centralizer conditions. 

\end{theo}

We present the proof in \refsec{secTheProof}. 
Note that the criterion of Theorem \ref{thMainResult} is entirely combinatorial. This clearly applies to the cone condition. Checking the centralizer condition under the assumption that the cone condition holds is also an entirely combinatorial procedure. Indeed, in this latter case Proposition \ref{lePara} below gives that $ C(\kk_{ss})\cap N(\nn)= C(\kk_{ss}) \cap N(C(\kk_{ss})\cap\nn)$, i.e. $ C(\kk_{ss})\cap N(\nn)$ is a parabolic subalgebra of $C(\kk_{ss})$. Therefore checking the centralizer condition reduces to checking the type of the root subsystem $Q\cap -Q\subset\Delta(\gG)$, where $Q:=\{\alpha\in\Delta(\gG)~|~\alpha\sperp \Delta(\kk)$, such that for all $\beta\in\Delta(\nn)$ with $\beta\sperp\Delta(\kk)$, either $\alpha+\beta\in\Delta(\nn)$, or $\alpha+\beta$ is not a root $\}$.

In the case when $\kk=\hh$ (i.e. $\lL$ is solvable), the cone condition is equivalent to the requirement that $\nn$ be the nilradical of a parabolic subalgebra containing $\hh$ (see \cite[Prop. 4]{PS:GenHarishChandra} and also Lemma \ref{leCentralizerAndNilradical} below). Furthermore, using Corollary 5.5 and Theorem 5.8 from \cite{PSZ}, it is not difficult to show that when $\gG$ is of type $A$, the cone condition holds if and only if $\nn$ is the nilradical of a parabolic subalgebra of $\gG$ which contains $\kk$. This is not the case in type $B,C$, and $D$. Here is an example for type $C$.

\begin{example}\label{exsp6} $\gG=\sP(6)$, $\Delta(\kk)=\{\pma2\eps_2,\pma2\eps_3\}$, $\Delta(\nn)=2\eps_1, \eps_1+ \eps_2, \eps_1- \eps_2 $.  A computation shows that $\hw_{ \bb\cap\kk } (\gG/\lL))$ = $\{\eps_3+ \eps_1,\eps_2+ \eps_3, \eps_3- \eps_1, -\eps_1+ \eps_2,-2\eps_1\}$. Thus the cone condition is satisfied, but there exists no parabolic subalgebra $\pp$ with $\nn=\nn_\pp$. Indeed, assume the contrary. Then there exists a vector $t\in h$ with $\alpha(t)>0$ for all $\alpha \in \Delta (\nn)$, and $\alpha(t)\leq 0$ for all other roots $\alpha$ of $\gG$. Therefore $\eps_2(t)=\eps_3(t)=0$,  $\eps_1(t)>0$, and $(\eps_1+\eps_3)(t)>0$. Contradiction.
\end{example}

To illustrate the cone condition in the non-solvable case, we present all non-solvable root subalgebras that fail the cone condition in types $B_3$ and $C_3$ (Table \ref{tableC3B3} below). These subalgebras are, up to conjugation, all non-solvable root subalgebras of infinite type in types $B_3$ and $C_3$. Indeed, the centralizer condition is trivially satisfied in type $C_3$. In type $B_3$, the centralizer condition holds for $\lL\nsimeq \hh$, as the root system $B_2$ is isomorphic to $C_2$. Up to conjugation, in $\so(7)$ (respectively, $\sP(6)$), there are 11 (respectively, 16) non-solvable root subalgebras that fail the cone condition, and 32 (respectively, 38) non-solvable root subalgebras that satisfy it. 

\begin{longtable}{r|l}
\multicolumn{2}{c}{ $ \mathfrak{g}\simeq \mathrm{so(7)}$} \\\hline\hline\begin{tabular}{r}$\Delta(\mathfrak{k})$ is of type $A_1$+$A_1$; \\ $\Delta^+(\mathfrak{k})=$ $\varepsilon_{1}$-$\varepsilon_{2}$, $\varepsilon_{1}$+$\varepsilon_{2}$\end{tabular} &
\begin{tabular}{l}$\Delta(\mathfrak{n})=$$\varepsilon_{1}$+$\varepsilon_{3}$, $\varepsilon_{2}$+$\varepsilon_{3}$, -$\varepsilon_{2}$+$\varepsilon_{3}$, -$\varepsilon_{1}$+$\varepsilon_{3}$, 
\end{tabular}

\\\hline\begin{tabular}{r}$\Delta(\mathfrak{k})$ is of type $A_1$; \\ $\Delta^+(\mathfrak{k})=$ $\varepsilon_{3}$\end{tabular} &
\begin{tabular}{l}$\Delta(\mathfrak{n})=$$\varepsilon_{1}$-$\varepsilon_{2}$, $\varepsilon_{1}$+$\varepsilon_{3}$, $\varepsilon_{1}$-$\varepsilon_{3}$, $\varepsilon_{1}$
\\$\Delta(\mathfrak{n})=$$\varepsilon_{1}$+$\varepsilon_{3}$, $\varepsilon_{1}$-$\varepsilon_{3}$, $\varepsilon_{1}$

\\\end{tabular}

\\\hline\begin{tabular}{r}$\Delta(\mathfrak{k})$ is of type $A_1$; \\ $\Delta^+(\mathfrak{k})=$ $\varepsilon_{1}$-$\varepsilon_{2}$\end{tabular} &
\begin{tabular}{l}$\Delta(\mathfrak{n})=$$\varepsilon_{3}$, -$\varepsilon_{2}$+$\varepsilon_{3}$, -$\varepsilon_{1}$+$\varepsilon_{3}$
\\$\Delta(\mathfrak{n})=$$\varepsilon_{1}$+$\varepsilon_{2}$, $\varepsilon_{1}$-$\varepsilon_{3}$, $\varepsilon_{2}$-$\varepsilon_{3}$, $\varepsilon_{1}$+$\varepsilon_{3}$, $\varepsilon_{2}$+$\varepsilon_{3}$

\\$\Delta(\mathfrak{n})=$$\varepsilon_{1}$+$\varepsilon_{2}$, $\varepsilon_{1}$-$\varepsilon_{3}$, $\varepsilon_{2}$-$\varepsilon_{3}$, -$\varepsilon_{2}$-$\varepsilon_{3}$, -$\varepsilon_{1}$-$\varepsilon_{3}$
\\$\Delta(\mathfrak{n})=$$\varepsilon_{1}$+$\varepsilon_{2}$, $\varepsilon_{1}$-$\varepsilon_{3}$, $\varepsilon_{2}$-$\varepsilon_{3}$, $\varepsilon_{1}$, $\varepsilon_{2}$

\\$\Delta(\mathfrak{n})=$$\varepsilon_{1}$+$\varepsilon_{2}$, $\varepsilon_{1}$-$\varepsilon_{3}$, $\varepsilon_{2}$-$\varepsilon_{3}$
\\$\Delta(\mathfrak{n})=$$\varepsilon_{1}$+$\varepsilon_{2}$, $\varepsilon_{1}$, $\varepsilon_{2}$

\\$\Delta(\mathfrak{n})=$$\varepsilon_{1}$-$\varepsilon_{3}$, $\varepsilon_{2}$-$\varepsilon_{3}$, -$\varepsilon_{2}$-$\varepsilon_{3}$, -$\varepsilon_{1}$-$\varepsilon_{3}$
\\$\Delta(\mathfrak{n})=$$\varepsilon_{1}$-$\varepsilon_{3}$, $\varepsilon_{2}$-$\varepsilon_{3}$

\\\end{tabular}

\\\hline\hline

\multicolumn{2}{c}{ $\mathfrak{g}\simeq \mathrm{sp(6)}$} \\\hline\hline\begin{tabular}{r}$\Delta(\mathfrak{k})$ is of type $A_1$; \\ $\Delta^+(\mathfrak{k})=$ -$\varepsilon_{2}$+$\varepsilon_{3}$\end{tabular} &
\begin{tabular}{l}$\Delta(\mathfrak{n})=$2$\varepsilon_{1}$, 2$\varepsilon_{3}$, 2$\varepsilon_{2}$, $\varepsilon_{2}$+$\varepsilon_{3}$
\\$\Delta(\mathfrak{n})=$2$\varepsilon_{1}$, -2$\varepsilon_{2}$, -2$\varepsilon_{3}$, -$\varepsilon_{2}$-$\varepsilon_{3}$

\\$\Delta(\mathfrak{n})=$$\varepsilon_{1}$+$\varepsilon_{3}$, $\varepsilon_{1}$+$\varepsilon_{2}$, 2$\varepsilon_{3}$, 2$\varepsilon_{2}$, $\varepsilon_{2}$+$\varepsilon_{3}$
\\$\Delta(\mathfrak{n})=$$\varepsilon_{1}$+$\varepsilon_{3}$, $\varepsilon_{1}$+$\varepsilon_{2}$

\\$\Delta(\mathfrak{n})=$2$\varepsilon_{3}$, 2$\varepsilon_{2}$, $\varepsilon_{2}$+$\varepsilon_{3}$
\\$\Delta(\mathfrak{n})=$$\varepsilon_{1}$-$\varepsilon_{2}$, $\varepsilon_{1}$-$\varepsilon_{3}$, -2$\varepsilon_{2}$, -2$\varepsilon_{3}$, -$\varepsilon_{2}$-$\varepsilon_{3}$

\\$\Delta(\mathfrak{n})=$$\varepsilon_{1}$-$\varepsilon_{2}$, $\varepsilon_{1}$-$\varepsilon_{3}$
\\$\Delta(\mathfrak{n})=$-2$\varepsilon_{2}$, -2$\varepsilon_{3}$, -$\varepsilon_{2}$-$\varepsilon_{3}$

\\
\end{tabular}

\\\hline\begin{tabular}{r}$\Delta(\mathfrak{k})$ is of type $A_1$; \\ $\Delta^+(\mathfrak{k})=$ 2$\varepsilon_{1}$\end{tabular} &
\begin{tabular}{l}$\Delta(\mathfrak{n})=$-$\varepsilon_{2}$+$\varepsilon_{3}$, 2$\varepsilon_{3}$, -2$\varepsilon_{2}$, $\varepsilon_{1}$+$\varepsilon_{3}$, -$\varepsilon_{1}$+$\varepsilon_{3}$
\\$\Delta(\mathfrak{n})=$-$\varepsilon_{2}$+$\varepsilon_{3}$, 2$\varepsilon_{3}$, $\varepsilon_{1}$+$\varepsilon_{3}$, -$\varepsilon_{1}$+$\varepsilon_{3}$

\\$\Delta(\mathfrak{n})=$-$\varepsilon_{2}$+$\varepsilon_{3}$, 2$\varepsilon_{3}$
\\$\Delta(\mathfrak{n})=$-$\varepsilon_{2}$+$\varepsilon_{3}$

\\$\Delta(\mathfrak{n})=$2$\varepsilon_{2}$, 2$\varepsilon_{3}$, $\varepsilon_{1}$+$\varepsilon_{2}$, -$\varepsilon_{1}$+$\varepsilon_{2}$
\\$\Delta(\mathfrak{n})=$2$\varepsilon_{2}$, 2$\varepsilon_{3}$

\\$\Delta(\mathfrak{n})=$2$\varepsilon_{2}$, $\varepsilon_{1}$+$\varepsilon_{2}$, -$\varepsilon_{1}$+$\varepsilon_{2}$
\\$\Delta(\mathfrak{n})=$2$\varepsilon_{2}$\quad .

\\\end{tabular}

\\
\multicolumn{2}{c}{Table  \refstepcounter{theo} \label{tableC3B3} \ref{tableC3B3} }
\\
\multicolumn{2}{c}{Non-solvable root subalgebras that fail the cone condition in types $B_3$ and $C_3$.}
\end{longtable}

\section{A sufficient condition for infinite type}
\label{secConeIntersections}
Let $\lL=\kk\crplus\nn$ be a root subalgebra, let $M$ be a $(\gG,\lL)$-module. For every root $\alpha\in\Delta(\gG)$ choose a non-zero vector $g^\alpha\in\gG^\alpha$ such that $[g^{\alpha},g^{-\alpha}]=h^\alpha$, where $h^{\alpha}$ is the element of $\hh$ for which $[h^{\alpha},g^\beta]=\langle\alpha,\beta \rangle g^{\beta}$ for all $\beta\in\Delta(\gG)$.

By Lie's theorem, there exists an $\bb\cap\lL$-singular vector $v$ in $M$. Suppose that there exist roots $\alpha_i\in\hw_{\bb\cap\kk}(\gG/\lL)$ and $\beta_i\in\Delta(\nn)$, as well as numbers $a_i,b_j\in\ZZ_{>0}$, such that the vectors of the form $\left((g^{-\beta_1})^{b_1}\dots (g^{-\beta_k})^{b_k}\right)^t$ $\left((g^{\alpha_1})^{a_1}\dots (g^{\alpha_l})^{a_l}\right)^t$  $\cdot v $ for $t\in\ZZ_{>0}$ have the following three properties. First, these vectors have the same $\hh$-weight; second, they are linearly independent; third, each of them projects naturally to a $\bb\cap\lL$-singular vector in an appropriate $\lL$-subquotient of $M$. If all three properties hold, then $M$ is a $(\gG,\lL)$-module of infinite type as the irreducible $\lL$-module with highest weight equal to the weight of $v$ has infinite multiplicity in $M$.

The above summarizes our approach for proving that the failure of the cone condition implies $\lL$ is a Fernando-Kac subalgebra of infinite type. The present section establishes that the three properties in question hold under an additional assumption. 
In section \ref{secBigTedious} we prove that this additional assumption is satisfied whenever the cone condition fails.

\label{secAbundanceCondition} \label{secMultiplicities}
\begin{defi}\label{defStronglyPerp}\label{defKmult}~
\begin{itemize}
\item Let $I$ be a set of roots and $\omega$ be a weight. We say that $\omega$ \emph{has a strongly orthogonal decomposition with respect to $I$} if there exist roots  $\beta_i\in I$ and positive integers $b_i$ such that $\omega=b_1\beta_1+\dots +b_k\beta_k$ and $\beta_i\sperp\beta_j$ for all $i,j$. 
\item Fix $\lL=\kk\crplus\nn\subset\gG$ with $\nn\subset\bb$. Let $\omega$ be a weight. We say that $\omega$ is \emph{two-sided with respect to $\lL$}, or simply \emph{two-sided}, if the following two conditions hold:
\begin{itemize}
\item $\omega\in \Cone_{\ZZ}(\hw_{\bb\cap\kk}(\gG/\lL))\cap\Cone_{\ZZ}(\Delta(\nn))\backslash\{0\}$, i.e. there exist $a_i\in\ZZ_{>0}$, $\alpha_i \in \hw_{\bb\cap\kk} (\gG/\lL)$, $b_i\in\ZZ_{>0}$ and $\beta_i\in\Delta(\nn)$ with 
\begin{equation}\label{eqRelation}
\omega=\sum_{i=1}^la_i\alpha_i=\sum_{i=1}^k b_i\beta_i;
\end{equation}
\item among all expressions for $\omega$ of type \refeq{eqRelation}, there exists one for which $[\gG^{\alpha_1},\nn]\subset\nn,\dots, [\gG^{\alpha_l},\nn]\subset\nn$. 
\end{itemize}
\item Let $\omega$ be a weight. If $\omega$ is both two-sided and has a strongly orthogonal decomposition with respect to $\Delta(\nn)$, we say that $\omega$ is \emph{$\lL$-strictly infinite}.
\item If for a given weight $\omega$ there exists a root subalgebra $\tT$ containing $\kk$, such that $\omega$ is $\lL'$-strictly infinite in $\tT$, where $\lL':=\lL\cap\tT=\kk\crplus(\tT\cap\nn)$, we say that $\omega$ is \emph{$\lL$-infinite}.
\end{itemize}
\end{defi}


\begin{lemma}\label{leNilradicalIsSeparate}
Given $\lL=\kk\crplus\nn\subset\gG$, there exists $h\in\hh$ such that $\gamma(h)=0$ for all $\gamma\in\Delta(\kk)$ and $\beta(h)>0$ for all $\beta\in\Delta(\nn)$.
\end{lemma}
\begin{proof}
Since $\hh\crplus\nn$ is a solvable Lie algebra, it lies in a maximal solvable (i.e. Borel) subalgebra; assume without loss of generality that this Borel subalgebra is $\bb$. Fix $h'\in\hh$ such that $\gamma(h')>0$ for all $\gamma\in \Delta^+(\gG)$. Let $h''\in\hh$ be defined by $\gamma(h''):=\gamma(h')$ for all $\gamma\in\Delta(\kk)$ and $\alpha(h'')=0$ for all weights $\alpha \in \Delta(\kk)^\perp$. Set $h:=h'-h''$. 

We claim that $h$ has the properties stated in the lemma. Indeed, let $\nn'\subset\nn$ be a $\kk$-submodule of $\nn$. Since $\gamma(h)=0$ for all $\alpha \in \Delta(\kk)$, the value $r:=\beta(h)$ is the same for all roots $\beta\in\Delta(\nn')$. Our statement is now equivalent to showing that $r>0$. Assume on the contrary that $r\leq 0$. Let the sum of the weights of $\Delta(\nn')$ be  $\lambda$, i.e. $\lambda:=\sum_{ \beta_i \in \Delta(\nn')}\beta_i$. Then $\lambda(h)= \#(\Delta(\nn'))r\leq 0$. On the other hand, the sum of the weights of a finite-dimensional $\kk_{ss}$-module always equals zero, i.e. $\lambda\in \Delta(\kk)^\perp$. Therefore $\lambda(h)=\lambda(h')-\lambda(h'')=\lambda(h')>0$, contradiction.
\end{proof}

For an arbitrary weight $\mu\in\hh^*$, denote by $L_{\mu}(\kk)$ (respectively, $L_{\mu}(\gG)$) the irreducible highest weight $\kk$-module (respectively, $\gG$-module) with $\bb\cap\kk$-highest (respectively, $\bb$-highest) weight $\mu$.
\begin{lemma}\label{leMultMain}
Let $\omega$ be a weight that has a strongly orthogonal decomposition with respect to $\Delta(\nn)$. Let $\lambda\in\hh^*$ be an arbitrary $\bb\cap\kk$-dominant and $\kk$-integral weight. Then there exists a number $t_0$ such that, for any $t>t_0$ and any $\gG$-module $M$ that has a $(\bb\cap\kk)\crplus\nn$-singular vector $v$ of weight $\lambda+t\omega$, it follows that $M$ has a $\kk$-subquotient in which there is a non-zero $\bb\cap\kk$-singular vector $\tilde w$ of weight $\lambda$.
\end{lemma}
Before we proceed with the proof we state the following.
\begin{corollary}
\label{corMultMain}
Let $\gG$, $\bb$, $\lL=\kk\crplus\nn$, $\lambda$ and $\omega$ be as above. Then there exists $t_0$ such that for any $t>t_0$ and any $(\gG,\kk)$-module that has a $(\bb\cap\kk)\crplus\nn$-singular vector $v\in M$ of weight $\lambda+t\omega$, it follows that $M$ has non-zero multiplicity of $L_\lambda(\kk)$. In particular, the existence of a $(\gG,\kk)$-module with the required singular vector implies that $\omega$ is $\bb\cap\kk$-dominant.
\end{corollary}
\textbf{Proof of \refle{leMultMain}.}Let $\nn^-$ be the subalgebra generated by the root spaces opposite to the root spaces of $\nn$. Let $b_1\beta_1+\dots +b_k\beta_k=\omega$ be a strongly orthogonal decomposition of $\omega$ with respect to $\Delta(\nn)$ (\refdef{defKmult}). Let $u:=(g^{\beta_1 })^{b_1}\dots (g^{ \beta_k})^{b_k} \in U(\nn)$ and $\bar u:= (g^{-\beta_1 })^{b_1}\dots (g^{ -\beta_k})^{b_k}$. 

Let $A$ be the linear subspace of $U(\nn^-)$ generated by all possible monomials $g^{-\gamma_1}\dots g^{-\gamma_k}$ that have strictly higher weight than $-t\omega$, where $-\gamma_i\in \Delta(\nn^-)$, in other words, $A:=\linspan \{g^{-\gamma_1}\dots g^{-\gamma_k}|\gamma_i\in\Delta(\nn), \sum\gamma_i\prec\omega\}$. Denote by $N$ the $\kk$-module generated by the vectors $\{A\cdot v\}$. To prove the lemma, we will show that the $\kk$-module $M/N$ has $\tilde w$  as a $\bb\cap\kk$-singular weight vector, where $\tilde w$ is the image in $M/N$ of $w:=\bar u^{t}\cdot v$. 

First, we will prove that $\tilde w$ is $\bb\cap\kk$-singular: indeed, $\nn^-$ is an ideal in the Lie subalgebra $\kk\crplus\nn^-$ and so $g^{\alpha} \bar u^t \in\left( \bar u^t g^{\alpha}+A\right)$ for all $\alpha \in \Delta^+(\kk)$; this, together with the fact that $v$ is $\bb\cap\kk$-singular, proves our claim. Second, we will prove that if $w$ is non-zero, then $w\notin N$ and therefore $\tilde w$ is non-zero. Indeed, the weight spaces of $N$ are a subset of the set 
\[
X:=\bigcup_{\substack{ \gamma\in\Cone_{\ZZ}(\Delta(\nn^-)) \\ \gamma\succ-t\omega }}(\lambda+t\omega+\gamma+ \linspan_\ZZ{ \Delta( \kk)} ).
\]
We claim that $X$ does not contain $\lambda$: indeed, choose $l\in h$ such that $\gamma(l)=0$ for all $\gamma\in\Delta(\kk)$ and $\beta'(l)>0$ for all $\beta'\in\Delta(\nn)$ (Lemma \ref{leNilradicalIsSeparate}). Therefore $-t\omega(l)\notin \{\mu(l)| \mu\in X\}$ and our claim is established.

To finish the proof of the lemma we are left to show that $w=\bar u^{t}\cdot v$ is non-zero, and this is the first and only place we will use the strongly orthogonal decomposition of $\omega$. To do that we will prove by direct computation that the vector $u^{t}\bar u^{t}\cdot v$ is a strictly positive multiple of $v$. For any $n\in\ZZ_{>0}$ we compute
\begin{eqnarray*}
(g^{\beta_i}) (g^{-\beta_i})^{n} \cdot v&=& \sum_{j=0}^{n-1} \langle\beta_i,-j\beta_i+\lambda+t\omega \rangle (g^{-\beta_i})^{n-1}\cdot v\\
&=&\left(nt\langle\beta_i,\omega\rangle+nt\langle\beta_i,\lambda\rangle - \frac{n(n-1)}{2}\langle\beta_i,\beta_i\rangle \right)(g^{-\beta_i})^{n-1}\cdot v\\
&=& \left(\langle\beta_i, \beta_i\rangle\left(b_int-\frac{n(n-1)}{2}\right)+ \langle\beta_i,\lambda\rangle\right) (g^{ - \beta_i })^{n-1}\cdot v .
\end{eqnarray*}
Therefore
\begin{eqnarray*}
(g^{\beta_i})^{t} (g^{-\beta_i})^{t}\cdot v&=&\prod_{k=0}^{t-1}\Bigg(  \langle\beta_i, \beta_i \rangle \left(b_i(t-k)t-\frac{(t-k)(t-k-1)}{2}\right)\\ &&~~~~~~~~~~+ \langle\beta_i,\lambda\rangle ~ \Bigg)\cdot v.
\end{eqnarray*}
Define $c_i(t,\lambda)$ to be the above computed coefficient of $ v$, in other words, set $ c_i(t,\lambda) v:= (g^{\beta_i})^{t-s} (g^{-\beta_i})^{t-s}\cdot v$. Since $b_i$ is a positive integer, using the explicit form of $c_i(t,\lambda)$, we see that for a fixed $\lambda$, $c_i(t,\lambda)>0$ for all large enough $t$. Using that $g^{\pma\beta_i}$ and $g^{\pma\beta_j}$ commute whenever $i\neq j$, we get immediately that $u^{t}\bar u^{t}\cdot v=\prod_i c_i(t,\lambda) v$, which proves our claim that $u^{t}\bar u^{t}\cdot v$ is  a positive multiple of $v$. Therefore $\bar u^{t}\cdot v$ cannot be zero, which completes the proof of the lemma.
$\Box$

\begin{example} 
Let us illustrate Lemma \ref{leMultMain} in the case when $\gG\simeq \sL(3)$ and $M$ is an irreducible $(\gG,\lL)$-module of finite type. Consider first the case $\kk=\hh$. If $\nn=\{0\}$, the statement of the Lemma is a tautology. If $\nn\neq \{0\}$, the lemma asserts that a certain weight space of $M$ is non-zero. As the $\hh$-characters of all simple $\sL(3)$-modules of finite type are known (see for instance, \cite[Section 7]{Mathieu}), the claim of the lemma is a direct corollary of this result. 

The only other possibility for $\kk\neq\gG$ is $\kk\simeq \sL(2)+\hh$. Then there are 2 options for $\lL$: $\lL=\kk$ or $\lL=\kk\crplus \nn$, where $\dim \nn=2$. For $\lL=\kk$ the lemma is a tautology as $\nn=\{0\}$. Consider the case when $\lL=\kk\crplus\nn$ with $\dim \nn=2$, i.e. the case when $\lL$ is a parabolic subalgebra with Levi component isomorphic to $\sL(2)$. Here, there are two options for $M$: $\dim M<\infty$, and $\dim M=\infty$. In both cases $\kk$ acts semisimply on $M$ and the lemma asserts the existence of certain $\bb\cap\kk$ singular vector in $M$. More precisely, let $\gamma_1:=\eps_1-\eps_2, \gamma_2 :=\eps_2 - \eps_3$ be the positive simple basis of $\Delta(\gG)$ with respect to $\bb$, and let $\Delta(\kk)=\{\pm\gamma_1\}$. Then Lemma \ref{leMultMain} claims that if $\lambda$ is $\bb\cap\kk$-dominant and integral, then $L_{\lambda+t\omega}(\gG)$ has a $\bb\cap\kk$-singular vector of weight $\lambda$ for all large enough $t$. Up to multiplication by a positive integer, there are two different options for picking the weight $\omega$ - either  $\omega=\gamma_1+\gamma_3$ or $\omega=\gamma_2+\gamma_3$.

\begin{itemize}
\item Suppose $\omega=\gamma_1+\gamma_3$.  Let $x(t)$ and $y(t)$ be functions of $t$ and $\lambda$, defined by $\lambda+t\omega = \frac{x(t)}{3} ( 2\gamma_1 + \gamma_2) +\frac{y(t)}{3} (\gamma_1 + 2 \gamma_2) $. The requirement that $\lambda$ be $\bb\cap\kk$-dominant forces $ t\leq x(t) $. Then the lemma states that there exists a constant $t_0$, such that for all $t_0\leq t\leq x(t)$ we have that $L_{\lambda+t\omega}(\gG)$ has a $\bb\cap\kk$-singular vector of weight $\lambda$. The reader can verify that for both infinite and finite-dimensional $M$, that the constant $t_0$ can be chosen to be zero.

\item Suppose $\omega=\gamma_2$. Then the lemma states that there exists a constant $t_0$, such that for all $t\geq t_0$ we have that $L_{\lambda+t\omega}(\gG)$ has a $\bb\cap\kk$-singular vector of weight $\lambda$. As the reader can verify, when $\dim M=\infty$, the statement of the Lemma holds for $t_0=0$; in the case that $M$ is finite-dimensional, one must pick $t_0\geq -\langle\lambda,\gamma_2\rangle$.
\end{itemize}
 
\end{example}

\begin{lemma}\label{leNotFK}
Suppose there exists an $\lL$-strictly infinite weight $\omega$. 
\begin{itemize}
\item[(a)] Any $(\gG,\lL)$-module $M$ for which any element in $\gG\backslash\lL$ acts freely is of infinite type over $\lL$.
\item[(b)] $\lL=\kk\crplus\nn$ is a Fernando-Kac subalgebra of infinite type. 
\end{itemize}
\end{lemma}
\begin{proof} 
As any irreducible strict $(\gG,\lL)$-module satisfies the conditions of (a), (a) implies (b); we will now show (a). Let $v_\lambda$ be a $(\bb\cap \lL)$-singular vector. 

Let $\omega:=\sum_{i=1}^l a_i\alpha_i= \sum_{i=1}^k b_i\beta_i$ be one decomposition \refeq{eqRelation}. Let $u^{ \bar \alpha}:= (g^{\alpha_1})^{a_1} \dots (g^{\alpha_l})^{a_l}\in U(\gG)$. The vector $v_{\lambda+ t\omega}:= (u^{ \bar\alpha} )^t\cdot v_\lambda$ is non-zero by the conditions of (a). We claim that $v_{\lambda+t\omega}$ is $\bb\cap \lL$-singular. Indeed, first note that since all $\alpha_i$ are $\kk \cap \bb$-singular, $v_{\lambda+t\omega}$ is $\bb\cap \kk$-singular. Second, let $g^{\beta} \in \gG^{\beta} \subset \nn$. By the second requirement for being two-sided we can commute $g^\beta$ with $u^{\bar \alpha}$ to obtain that $g^{\beta}(u^{\bar\alpha})^n \in U(\qq)  a$, where $a\in U(\nn)$ is an element with no constant term and $\qq$ is the Lie subalgebra generated by $\gG^{\alpha_1},\dots, \gG^{ \alpha_k }$. Since $a\cdot v_{\lambda}=0$, we get $g^{\beta}\cdot v_{\lambda+t\omega}=0$, which proves our claim.

All $v_{\lambda+t\omega}$ are linearly independent since they have pairwise non-coinciding weights. Let $M_t$ be the $\gG$-submodule of $M$ generated by $v_{\lambda+t\omega}$ and let $M'$ be the sum of the $M_t$'s as $t$ runs over the non-negative integers. \refcor{corMultMain} shows that the $\kk$-module $L_{\lambda }( \kk )$ has non-zero multiplicity in $M_t$ for all large enough $t$. Consider the vectors $\bar u^t\cdot v_{\lambda+t\omega}$ generating the $\kk$-subquotients isomorphic to $L_{\lambda }( \kk )$, where $\bar u$ is defined as in the proof of \refle{leMultMain}. Let $A_t$ be the linear subspace of $U(\nn^-)$ generated by all possible monomials $g^{-\gamma_1}\dots g^{-\gamma_k}$ that have strictly higher weight than $-t\omega$, where $\gamma_i\in\Delta(\nn)$. Let $N$ be the $\kk$-submodule generated by the vectors $\displaystyle \bigcup_{t> t_0}A_t \cdot v_{\lambda+t\omega}$, where $t_0$ is the number given by Lemma \ref{leMultMain}. Just as in the proof of \refle{leMultMain} we see that each vector $\bar u^t\cdot v_{\lambda+t\omega}$ is not in $N$ and is the image of a $\bb\cap\kk$-singular vector in the quotient $ M'/N $. 

We will now prove that $\bar u^t\cdot v_{\lambda+t\omega}$ are linearly independent. Indeed, let $u$ be defined as in the proof of \refle{leMultMain}. Now take a linear dependence $0=\sum_{i=1}^N c_i \bar u^{t_i}\cdot v_{\lambda+t_i\omega}$ such that $t_N\geq t_1, \dots, t_N\geq  t_{N-1}$ and apply $u^{t_N}$ to both sides. As the computations in the proof of \refle{leMultMain} show, $u^{t_N}$ kills all but the last summand; therefore the last summand has coefficient $c_N=0$. Arguing in a similar fashion for the remaining summands, we conclude that the starting linear dependence is trivial. This shows that the $\kk$-module $L_{\lambda }( \kk )$ has infinite multiplicity in the $\kk$-module $M'$. We conclude that $M$ has infinite type over $\kk$, hence, by \cite[Theorem 3.1]{PSZ}, $M$ has also infinite type over $\lL$. 
\end{proof}

\section{Existence of $\lL$-infinite weights}\label{secBigTedious}
\subsection{Existence of two-sided weights}\label{secTwoSided}

\begin{lemma}\label{lestreakOfRoots} Let $\alpha_1,\dots,\alpha_k,\alpha_{k+1},\gamma$ be vectors of a root system such that $\alpha_1+ \dots+ \alpha_k + \alpha_{k+1}+\gamma$ is a root different from $\gamma$ or is equal to zero, and $\alpha_i+\gamma$ is neither a root nor zero for $i=1,\dots,k$. Then $\alpha_{k+1}+\gamma$ is a root or zero.
\end{lemma}
\begin{proof}
We will establish the lemma only for an irreducible root system; the case of a reducible root system is an immediate corollary which we leave to the reader. For $G_2$ the statement is a straightforward check, so assume in addition that the root system is not of type $G_2$. 

Assume the contrary to the statement of the lemma. Let 
\begin{equation}\label{eqLeRootTricks}
\alpha_1+\dots+\alpha_k+\alpha_{k+1}+\gamma=\delta.
\end{equation} 
Then $\langle\beta,\gamma\rangle\geq 0$, $\langle\alpha_i,\gamma\rangle\geq 0$. Apply $\langle \bullet,\gamma\rangle$ to both sides of (\ref{eqLeRootTricks}). We get $\langle\alpha_1, \gamma\rangle+ \dots+\langle \alpha_k , \gamma \rangle+\langle\alpha_{k+1}, \gamma \rangle + \langle \gamma, \gamma \rangle = \langle \delta,\gamma\rangle$. If $\delta=0$, we immediately get that $\langle\gamma,\alpha_i\rangle<0$ for some $i$ and the statement of the lemma holds as the sum of two roots with negative scalar product is always a root. Therefore we can suppose until the end of the proof that $\delta\neq 0$. 

Since $\langle\alpha_i,\gamma\rangle\geq 0$ and $\delta\neq \gamma$, we must have $\langle \gamma, \alpha_1 \rangle = \dots= \langle \gamma, \alpha_k \rangle= \langle \gamma, \alpha_{k+1} \rangle= 0$ and $\langle \gamma, \gamma \rangle= \langle \delta, \gamma \rangle $. Since $\delta \neq \gamma$ by the conditions of the lemma, the only way for this to happen is to have that $\gamma$ is a short and $\delta$ is long, which gives the desired contradiction in types $A$, $D$ and $E$. Suppose now the given root system is of type $C$. Then without loss of generality we can assume that $\delta = 2\eps_1$ and $\gamma= \eps_1 + \eps_2$. But then there must be a summand on the left-hand side of (\ref{eqLeRootTricks}) which cancels the $+\eps_2$ term of $\gamma$. None of the $\alpha_i$'s have a $-\eps_2$ term (since $\alpha_i+\gamma$ is not a root) and therefore $\alpha_{k+1}+\gamma$ is a root, contradiction. Suppose next that the given root system is of type $B$. Then without loss of generality $\gamma$ can be assumed to be $\eps_1$ and $\delta$ to be $\eps_1+\eps_2$. Clearly  $\alpha_1+ \dots+\alpha_k +\alpha_{k+1} +\eps_1 = \eps_1+\eps_2$ wouldn't be possible if all $\alpha_i$'s and $\alpha_{k+1}$ were long. Therefore one of them is short, which implies that this root plus $\gamma$ is a root, contradiction.

Suppose finally that the given root system is of type $F_4$. Pick a minimal relation (\ref{eqLeRootTricks}) that contradicts the statement of the lemma, i.e. one with minimal number of $\alpha_i$'s. This number must be at least 3, since otherwise this relation would generate a root subsystem of rank 3 or less and this is impossible by the preceding cases. We claim that for all $i,j$, $\alpha_{ij}:=\alpha_i+\alpha_j$ is not a root or zero. Indeed, assume the contrary. If $\alpha_{ij}+\gamma= \alpha_i+ \alpha_j +\gamma$ is not a root or zero we could replace $\alpha_i+\alpha_j$ by $\alpha_{ij}$ in contradiction with the minimality of the initial relation. Therefore $\alpha_{ij}+\gamma= \alpha_i+ \alpha_j +\gamma$ is a root or zero, and since the three roots $\alpha_i,\alpha_j,\gamma$ generate a root subsystem of rank at most 3, the preceding cases imply that at least one of $\alpha_i+\gamma$ and $\alpha_j+\gamma$ is a root, contradiction.

So far, for all $i,j$, we established that $\alpha_i+\alpha_j$ is not a root or zero; therefore $\langle\alpha_i,\alpha_j\rangle\geq 0$ for all $i,j$.
Taking $\langle\alpha_1,\bullet \rangle$ on each side of  $\alpha_1+ \dots+ \alpha_k +\alpha_{k+1}+\gamma=\delta$ we see that $2\leq\langle\alpha_1,\alpha_1\rangle\leq \langle\alpha_1,\delta\rangle$. Therefore $\delta-\alpha_1$ is a root or zero, and transferring $\alpha_1$ to the right-hand side we get a shorter relation than the initial one. Contradiction.
\end{proof}

\begin{defi}\label{defMinRelation}
For a relation \refeq{eqRelation} we define \emph{the length of the relation} to be $\sum_i a_i$. We define a relation \refeq{eqRelation} to be \emph{minimal} if its length is minimal, there are no repeating summands on either side, and no two $\beta_i$'s sum up to a root. 
\end{defi}

\textbf{Remark.} Any relation of minimal length can be transformed to a minimal relation by combining the repeating summands on both sides and by replacing the $\beta_i$'s in \refeq{eqRelation} that sum up to roots by their sums. If in addition the initial relation of minimal length corresponds to a two-sided weight, Lemma \ref{lestreakOfRoots} implies that the resulting minimal relation again corresponds to a two-sided weight.

\begin{prop}\label{propabundanceHolds}
Let the cone condition fail. Then there exists a minimal relation (\ref{eqRelation}) corresponding to a two-sided weight $\omega$.
\end{prop}
\begin{proof}
The failure of the cone condition is equivalent to the existence of a relation \refeq{eqRelation}. Pick a minimal such relation. Assume that the weight arising in this way is not two-sided. Together with the minimality of the relation this implies that for one of the $\alpha_i$'s, say $\alpha_1$, there exist roots $\beta'\in\Delta(\nn)$ and $\delta\in\Delta(\gG)\backslash\Delta(\nn)$ such that $\delta= \alpha_{1}+\beta'$.

We claim that $\delta\notin\Delta(\lL)$. Indeed, assume on the contrary that $\delta\in\Delta(\kk)$. Then $\beta'-\delta$ is a root, and therefore lies in $\Delta(\nn)$. We get the relation $(a_1-1)\alpha_1+a_2\alpha_2+ \dots+ a_l\alpha_l = b_1\beta_1 +\dots+b_k\beta_k+(\beta'-\delta)$ which is shorter than the initial relation, contradiction. 

Now suppose that $\delta$ is not $\bb\cap\kk$-singular. Therefore there exists $\gamma_1\in\Delta^+(\kk)$ such that $\delta_1 :=\delta + \gamma_1$ is a root. If $\delta_1$ is not singular, continue picking in a similar fashion roots $\gamma_2,\dots , \gamma_s \in \Delta^+(\kk)$, such that $\delta_t= \delta+ \gamma_1+ \dots +\gamma_t$ is a root for any $t\leq s$. Since this process must be finite, $\delta_s$ is $\bb\cap\kk$-singular for some $s$. As $\alpha_1$ is $\bb\cap\kk$-singular, $\alpha_1 +\gamma_1$ is not a root.  Apply now \refle{lestreakOfRoots} to $\delta_1=\alpha_1+\beta'+\gamma_1$ to get that $\beta'':=\beta'+\gamma_1$ is a root. Therefore $\beta''\in\Delta(\nn)$. Arguing in a similar fashion, we obtain that $\beta''':=\beta''+ \gamma_2$ is a root of $\nn$, and so on. Finally, we obtain $\beta^{(s+1)}:=\beta'+\gamma_1+\dots +\gamma_{s}\in\Delta(\nn)$ and so we get a new relation \refeq{eqRelation}:
\begin{equation}\label{eqminMaxRel}
(a_1-1)\alpha_1+ \delta_s+a_2\alpha_2+\dots+a_l\alpha_l=\beta_1+\dots +\beta_k+\beta^{(s+1)}.
\end{equation}
We can reduce \refeq{eqminMaxRel} so that no two $\beta$'s add to a root (replace any such pairs by their sum) and so that if $\delta_s=\alpha_i$ for some $i$ then $\delta_s+a_i\alpha_i$ is replaced by $(a_i+1)\alpha_i$. 

This reduction of (\ref{eqminMaxRel}) is a minimal relation. If this relation does not yield a two-sided weight, one applies the procedure again and obtains a new minimal relation, and so on. As this process adds vectors from $\Delta(\nn)$ to the right-hand side of the relation, while the length of the left-hand side remains constant, the process must be finite (cf. Lemma \ref{leNilradicalIsSeparate}). Therefore there exists a minimal relation corresponding to a two-sided weight.
\end{proof}
\subsection{From two-sided to $\lL$-infinite weights}
In the remainder of this section we prove that the failure of the cone condition implies the existence of an $\lL$-infinite weight: our proof is mathematical for the classical Lie algebras and $G_2$ and uses a computer program for the exceptional Lie algebras $F_4$, $E_6$ and $E_7$.

For the classical Lie algebras, our scheme of proof can be summarized as follows. First, we classify all minimal relations (\ref{eqRelation}). It turns out by direct observation that whenever the cones intersect, the minimal relations (\ref{eqRelation}) are always of length 2, in particular this minimal length does not depend on the rank of the root system. In type $A$ this was discovered in \cite{PSZ}. In types $A$, $B$ and $D$, a direct inspection of all minimal relations shows that each of them possesses a strongly orthogonal decomposition with respect to $\Delta(\nn)$. Since at least one minimal relation must be two-sided by Proposition \ref{propabundanceHolds}, we obtain the existence of an $\lL$-strictly infinite weight. 

In type $C$ we do not have that all minimal relations possess a strongly orthogonal decomposition. However, the ``discrepancy'' is small - there is only one minimal relation (\ref{eqRelation}) without such a decomposition. In this particular case, we exhibit a root subalgebra $\tT$ containing $\kk$ such that $\tT$ has an $\lL\cap\tT$-strictly infinite weight, i.e. there is an $\lL$-infinite weight.

The proof for the exceptional Lie algebras uses a mixture of combinatorics and computer brute force. If $C(\kk_{ss})\cap\nn$ is not the nilradical of a parabolic subalgebra of $C(\kk_{ss})$, we prove in \refle{leCentralizerAndNilradical} that an $\lL$-infinite weight always exists, involving only roots of the root system of $C(\kk_{ss})$. Then, using our computer program, we enumerate up to $\gG$-automorphisms all remaining cases - i.e. the root subalgebras for which $C(\kk_{ss})\cap\nn$ is the nilradical of a parabolic subalgebra of $C(\kk_{ss})$ containing $C(\kk_{ss})\cap \hh$. This direct computation shows the existence of $\lL$-strictly infinite weights in types $E_6$ and $E_7$. In type $F_4$, our program fails to exhibit an $\lL$-strictly infinite weight for only one (unique up to $\gG$-automorphism) choice of $\lL$; in this case we give an argument similar to that in the special case in type $C$.

In order to enumerate all possible subalgebras $\kk$ we use the classification of reductive root subalgebras given in the fundamental paper \cite{Dynkin:Semisimple}. The list of possible proper root subalgebras $\kk$ is very short (it contains respectively 22, 19, 42, and 75 entries for $F_4$, $E_6$, $E_7$ and $E_8$\arxivVersion{, see section \ref{secRootSATables} in the appendix}). For a fixed $\kk$, we use all automorphisms of $\Delta(\gG)$ which preserve $\Delta(\bb\cap\kk)$\arxivVersion{ (see section \ref{secCardinalitiesIsos} in the appendix)} in order to generate only pairwise non-conjugate subalgebras $\lL$ and thus further decrease the size of the computation.


The following is an observation that is helpful in the proof of Lemma \ref{leBigTedious} (\textit{cf.} \cite[Lemma 5.4]{PSZ}).
\begin{lemma}\label{leminrelationgeneral}
Let the cone condition fail and let us have a minimal relation \refeq{eqRelation}. Then
\begin{itemize}
\item[(a)] The relation has the form 
\[
\alpha_1+\alpha_2=\beta_1,
\] 
or 
\item[(b)] $\alpha_i+\alpha_j$ is not a root for all $i,j$.
\end{itemize}
\end{lemma}
\begin{proof}
Pick a minimal relation \refeq{eqRelation}. Suppose there exist indices $i,j$ such that $\gamma:=\alpha_i+\alpha_j$ is a root. We claim that $\gamma\in \Delta(\nn)$. Indeed, assume the contrary. First, suppose $\gamma\in\Delta^-(\kk)$. Then $\alpha_i$ and $\alpha_j$ would both fail to be $\bb\cap\kk$-singular. 

Second, suppose $\gamma\in\Delta^+(\kk)$. We prove that $\alpha_3+\dots+\alpha_l=\beta_1+\dots +\beta_k-\gamma$ is a shorter relation than \refeq{eqRelation}. Indeed, $\beta_1+\dots +\beta_k-\gamma$ is clearly non-zero (positive linear combination of elements of $\Delta(\nn)$ cannot be in the span of the roots of the semisimple part). By the $\bb\cap\kk$-singularity of the $\alpha_i$'s,   $\langle\gamma, \beta_1+\dots+ \beta_k\rangle =\langle\gamma,\alpha_1 +\dots +\alpha_k\rangle =\langle \gamma, \gamma \rangle +\langle \gamma, \alpha_3 +\dots + \alpha_k \rangle>0$ and therefore, for some index $i$,  $\langle\gamma,\beta_i\rangle>0$. This shows that $\beta_i-\gamma$ is a root, which therefore belongs to $\Delta(\nn)$. Contradiction.

Third, suppose $\gamma\notin\Delta(\gG)\backslash\Delta(\lL)$. Then $\gamma$ is $\bb\cap\kk$-singular - if $\gamma+ \delta = \alpha_1 + \alpha_2+\delta$ were a root for some $\delta\in\Delta^+(\kk)$, then \refle{lestreakOfRoots} would imply that $\alpha_1+\delta$ is also a root. Therefore we can shorten the relation \refeq{eqRelation} by replacing $\alpha_1+\alpha_2$ by $\gamma$, and the obtained relation is non-trivial since the right-hand side is not zero. Contradiction.

Therefore $\gamma\in \Delta(\nn)$, and our lemma is proved.
\end{proof}

\subsection{Minimal relations \refeq{eqRelation} in the classical Lie algebras}
The following lemma describes all minimal relations (\ref{eqRelation}) up to automorphisms of $\Delta(\gG)$.
\begin{lemma}\label{leBigTedious}
Let $\gG\simeq\so(2n),\so(2n+1),\sP(2n)$. Suppose $\lL=\kk\crplus\nn$ does not satisfy the cone condition. 
\begin{itemize}
\item A minimal relation \refeq{eqRelation} has length 2 (Definition \ref{defMinRelation}).
\item All possibilities for minimal relations \refeq{eqRelation}, up to an automorphism of $\Delta(\gG)$, are given in the following table.
\end{itemize}
\end{lemma}
\noindent \begin{longtable}{rcl | c |c c}
\multicolumn{3}{c|}{$\omega$} & \begin{tabular}{p{4.5cm}}Scalar products. All non-listed scalar products are zero. All roots, unless stated otherwise, are assumed long in types $B,D$ and short in type $C$.\end{tabular}&  \multicolumn{2}{p{2cm}}{The roots from the relation generate }
\endhead
\hline
\multicolumn{5}{c}{$\gG\simeq\so(2n)$}
\\\hline
$\alpha_1+\alpha_2$ &$=$&$\beta_1$&
\begin{tabular}{l}
$\langle\alpha_1, \beta_1 \rangle$ $=$ $\langle\alpha_2,\beta_1\rangle$ $=1$\\
$\langle\alpha_1,\alpha_2\rangle$ $=-1$
\end{tabular}&
 $A_2$&
 \refstepcounter{equation} \label{eqDcontainsA2'} (\ref{eqDcontainsA2'})\\\hline
$\alpha_1+\alpha_2$ &$=$&$\beta_1+\beta_2$&
\begin{tabular}{l}
$\langle\alpha_1,\beta_1\rangle$ $=$ $\langle\alpha_1,\beta_2\rangle$ $=$\\ 
$\langle\alpha_2,\beta_1\rangle$ $=$ $\langle\alpha_2,\beta_2\rangle$ $=1$ 
\end{tabular}&
\begin{tabular}{l}
 $A_3\subset A_4$,\\ $n\geq 5$ 
\end{tabular}& \refstepcounter{equation} \label{eqDcontainsA3'} (\ref{eqDcontainsA3'})
\\\hline
$\alpha_1+\alpha_2$ &$=$&$\beta_1+\beta_2+\beta_3$&
\begin{tabular}{l}
$\langle\alpha_1,\alpha_2\rangle$ $=$  $\langle\alpha_1,\beta_1\rangle$ $=$\\
$\langle\alpha_1,\beta_2\rangle$ $=$ $\langle\alpha_1,\beta_3\rangle$ $=$ \\
$\langle\alpha_2,\beta_1\rangle$ $=$ $\langle\alpha_2,\beta_2\rangle$ $=$ \\
$\langle\alpha_2,\beta_3\rangle$ $=1$
\end{tabular}& 
$D_4$&
\refstepcounter{equation} \label{eqDcontainsD4'} (\ref{eqDcontainsD4'})
\\\hline
$2\alpha_1$ &$=$&$\beta_1+\beta_2+\beta_3+\beta_4$&
\begin{tabular}{l}
$\langle\alpha_1,\beta_1\rangle$ $=$ $\langle\alpha_1,\beta_2\rangle$ $=$ \\
$\langle\alpha_1,\beta_3\rangle$ $=$ $\langle\alpha_1,\beta_4\rangle$ $=1$ 
\end{tabular}&
$D_4$&
\refstepcounter{equation} \label{eqDcontainsD4''} (\ref{eqDcontainsD4''})\\\hline
$\alpha_1+\alpha_2$ &$=$&$\beta_1+\beta_2$&
\begin{tabular}{l}
$\langle\alpha_1,\beta_1\rangle$ $=$ $\langle\alpha_1,\beta_2\rangle$ $=$ \\
$\langle\alpha_2,\beta_1\rangle$ $=$ $\langle\alpha_2,\beta_2\rangle$ $=1$ 
\end{tabular}&
 $A_3$\footnote{\cite[Table 9]{Dynkin:Semisimple} uses the notation ``$D_3$'' for such subalgebras. $D_3$ is defined as a root subsystem of type $A_3$ of root system of type $B$ or $D$, which cannot be extended to a root subsystem of type $A_4$.}
&
\refstepcounter{equation} \label{eqDcontainsD3'} (\ref{eqDcontainsD3'}) \\\hline
\multicolumn{5}{c}{$\gG\simeq\so(2n+1)$}\\\hline
\multicolumn{3}{c|}{\begin{tabular}{c} all relations \\ listed for $\so(2n)$\end{tabular}}& - & - \\\hline
$\alpha_1+\alpha_2$ &$=$&$\beta_1$&
\begin{tabular}{l}
$\langle\alpha_1, \beta_1 \rangle$ $=$ $\langle\alpha_2,\beta_1\rangle$ $=1$,\\
$\|\alpha_1\|=\|\alpha_2\|=1 $ 
\end{tabular}&
 $B_2$ &
 \refstepcounter{equation} \label{eqBcontainsB2'} (\ref{eqBcontainsB2'})\\\hline
$\alpha_1+\alpha_2$ &$=$&$\beta_1$&  
\begin{tabular}{l}
$\langle\alpha_1, \alpha_2 \rangle$ $=-1$, $\langle\alpha_2,\beta_1\rangle$ $=1$, \\
$\|\alpha_1\|=\|\beta_1\|=1 $ 
\end{tabular}&
 $B_2$&
\refstepcounter{equation} \label{eqBcontainsB2''} (\ref{eqBcontainsB2''}) \\\hline
$2\alpha_1$ &$=$&$\beta_1+\beta_2$&  
\begin{tabular}{l}
$\|\alpha_1\|=1, \langle\alpha_1,\beta_1\rangle=\langle\alpha_1,\beta_2\rangle=1$ 
\end{tabular}&
 $B_2$&
\refstepcounter{equation} \label{eqBcontainsB2'''} (\ref{eqBcontainsB2'''}) \\\hline
$\alpha_1+\alpha_2$ &$=$&$2\beta_1$&  
\begin{tabular}{l}
$\|\beta_1\|=1, \langle\alpha_1,\beta_1\rangle=\langle\alpha_2,\beta_1\rangle=1$ 
\end{tabular}&
 $B_2$&
 \refstepcounter{equation} \label{eqBcontainsB2''''} (\ref{eqBcontainsB2''''}) \\\hline
$\alpha_1+\alpha_2$ &$=$&$\beta_1+\beta_2$&  
\begin{tabular}{l}
$\langle\alpha_1, \beta_1 \rangle=$  $\langle\alpha_2,\beta_1\rangle=$ $\langle\alpha_2,\beta_2\rangle$ $=1$, \\
$\|\alpha_1\|=\|\beta_2\|=1 $ 
\end{tabular}&
 $B_3$&
\refstepcounter{equation} \label{eqBcontainsB3'} (\ref{eqBcontainsB3'}) \\\hline
$\alpha_1+\alpha_2$ &$=$&$2\beta_1+\beta_2$&  
\begin{tabular}{l}
$\langle\alpha_1, \alpha_2 \rangle=$ $\langle\alpha_1,\beta_1\rangle=$ $\langle\alpha_2,\beta_1\rangle$ $=1$, \\
$\langle\alpha_1,\beta_2\rangle=$ $\langle\alpha_2,\beta_2\rangle=$ $\|\beta_1\|$ $=1$, 
\end{tabular}&
 $B_3$&
 \refstepcounter{equation} \label{eqBcontainsB3''} (\ref{eqBcontainsB3''}) \\\hline
$2\alpha_1$ &$=$&$2\beta_1+\beta_2+\beta_3$&  
\begin{tabular}{l}
$\langle\alpha_1, \beta_1 \rangle=$ $\langle\alpha_1,\beta_2\rangle=$ $\langle\alpha_1,\beta_3\rangle$ $=1$, \\
$\|\beta_1\|=1 $ 
\end{tabular}&
$B_3$&
\refstepcounter{equation} \label{eqBcontainsB3'''} (\ref{eqBcontainsB3'''}) \\\hline
\multicolumn{5}{c}{$\gG\simeq\sP(2n)$}\\\hline
$\alpha_1+\alpha_2$ &$=$&$\beta_1$&
\begin{tabular}{l}
$\langle\alpha_1, \beta_1 \rangle$ $=$ $\langle\alpha_2,\beta_1\rangle$ $=1$\\
$\langle\alpha_1,\alpha_2\rangle$ $=-1$
\end{tabular}&
$A_2$&
\refstepcounter{equation} \label{eqCcontainsA2'} (\ref{eqCcontainsA2'})\\\hline
$\alpha_1+\alpha_2$ &$=$&$\beta_1$&
\begin{tabular}{l}
$\langle\alpha_1, \beta_1 \rangle$ $=$ $\langle\alpha_2,\beta_1\rangle$ $=2$,\\
$\|\beta_1\|=2$
\end{tabular}&
$C_2$&
\refstepcounter{equation} \label{eqCcontainsC2'} (\ref{eqCcontainsC2'}) \\\hline
$\alpha_1+\alpha_2$ &$=$&$2\beta_1$&
\begin{tabular}{l}
$\langle\alpha_1,\beta_1\rangle=\langle\alpha_2,\beta_1\rangle$ $=1$ \\
$\|\alpha_1\|=\|\alpha_2\|$ $=2$
\end{tabular}& 
$C_2$&
\refstepcounter{equation} \label{eqCcontainsC2''} (\ref{eqCcontainsC2''}) \\\hline
$\alpha_1+\alpha_2$ &$=$&$\beta_1$&
\begin{tabular}{l}
$\langle\alpha_1,\beta_1\rangle=2$, $\langle\alpha_2,\alpha_1\rangle=-1$ \\
$\|\alpha_1\|$ $=2$
\end{tabular}& 
$C_2$&
\refstepcounter{equation} \label{eqCcontainsC2'''} (\ref{eqCcontainsC2'''}) \\\hline
$\alpha_1+\alpha_2$ &$=$&$\beta_1+\beta_2$&
\begin{tabular}{l}
$\langle\alpha_1,\beta_1\rangle$ $=$ $\langle\alpha_1,\beta_2\rangle$ $=$ \\
$\langle\alpha_2,\beta_1\rangle$ $=$ $\langle\alpha_2,\beta_2\rangle$ $=1$ 
\end{tabular}& 
 $A_3$&
 \refstepcounter{equation} \label{eqCcontainsA3'} (\ref{eqCcontainsA3'})\\\hline
$\alpha_1+\alpha_2$ &$=$&$\beta_1+\beta_2$&
\begin{tabular}{l}
$\langle\alpha_2,\beta_1\rangle$ $=$ $\langle\alpha_2,\beta_2\rangle$ $=2$,\\
$\langle\alpha_1,\beta_1\rangle=\langle\alpha_1,\beta_1\rangle$ $=$ \\
$\langle\alpha_1,\beta_2\rangle=1$,  $\langle\beta_1,\beta_2\rangle=1$, \\
$\|\alpha_2\|=2$ 
\end{tabular}& 
 $C_3$&
 \refstepcounter{equation} \label{eqCcontainsC3'} (\ref{eqCcontainsC3'})\\\hline
$\alpha_1+\alpha_2$ &$=$&$\beta_1+\beta_2$&
\begin{tabular}{l}
$\langle\alpha_1,\beta_2\rangle$ $=$ $\langle\alpha_1,\beta_2\rangle$ $=2$, \\
$\langle\alpha_1,\alpha_2\rangle=\langle\alpha_1,\beta_1\rangle$ $=$ \\
$\langle\alpha_2,\beta_1\rangle=1$, $\|\beta_2\|=2$ 
\end{tabular}& 
 $C_3$&
 \refstepcounter{equation} \label{eqCcontainsC3''} (\ref{eqCcontainsC3''})\\\hline
\end{longtable}

\begin{proof}Pick a minimal relation \refeq{eqRelation} of the form $\omega:=a_1\alpha_1+\dots+a_l\alpha_l=b_1\beta_1+\dots+b_k\beta_k$ (see Definition \ref{defMinRelation}).

Throughout this proof, we will use the informal expression ``$\pm\eps_{i}$ appears with a positive (resp. non-positive) coefficient in the weight $\omega$'' to describe the $\pm\eps_i$-coordinate of $\omega$ in the basis $\{\eps_1,\dots, \eps_{i-1},\pmf{i}\eps_i, \eps_{i+1}, \dots, \eps_n\}$. 

\subsubsection{$\gG\simeq\so(2n)$}
\smallskip\noindent Case 1. There exists an index $i$, such that $\alpha_i =\pmf{j_1} \eps_{j_1} +(\pmf{j_2} \eps_{j_2})$, $j_1\neq j_2$ and both $\pmf{j_1} \eps_{j_1}$ and $\pmf{j_2}\eps_{j_2}$ appear with a positive coefficient in $\omega$. Without loss of generality we may assume $i=1$ and $\alpha_1= \eps_{1}+\eps_{2}$. Therefore there exist $\beta_{1}$ and $\beta_{2}$ on the right-hand side of the relation with $\beta_1= \eps_{1} +(\pmf{j_3}\eps_{j_3})$ and $\beta_2=\eps_{2}+(\pmf{j_4} \eps_{j_4})$. The minimality of the relation implies $\{1, 2\}\cap\{j_3,j_4\}=\emptyset $. The latter allows us to assume without loss of generality that $\beta_1= \eps_{1} +\eps_{3}$.
\[
\underbrace{\eps_{1}+\eps_{2}}_{\alpha_1}+\dots= \underbrace{\eps_{1}+\eps_{3}}_{\beta_1}+\underbrace{ \eps_{2} +\eps_{j_3}}_{\beta_2}+\dots .
\]
We will now prove $j_4\neq 3$. 

Assume on the contrary that $3=j_4$. As the relation is minimal, the choice of $\pmf{j_4}$ sign must be such that $\eps_{3}= \pmf{j_4} \eps_{j_4}$. The minimality of the relation implies that there can be no cancellation of the weight $\eps_{3}$ on the right-hand side. Therefore on the left-hand side there exists a root, say $\alpha_2$, such that $\alpha_2=\eps_{3}+(\pmf{j_5}\eps_{j_5})$, $j_5\neq 3$. The minimality of the relation implies that in addition $j_5\neq 1,2$. Thus we can assume without loss of generality that $j_5=5$ and $\alpha_2=\eps_3+\eps_5$. So far, the assumption that $j_4=3$ implies that the relation has the form
\begin{equation}\label{eqConesso(2n)}
\underbrace{\eps_{1}+\eps_{2}}_{\alpha_1}+\underbrace{ \eps_{3}+\eps_{5}}_{\alpha_{2}}+\underbrace{\dots}_{\gamma}= \underbrace{\eps_{1}+\eps_{3}}_{\beta_1}+\underbrace{ \eps_{2} +\eps_{3}}_{\beta_2}+\underbrace{\dots}_{\delta},
\end{equation}
where $\gamma$ and $\delta$ denote the omitted summands. Suppose at least one of the roots $\eps_{1}+\eps_{5}$ and $\eps_{2} +\eps_{5}$ belongs to $\Delta(\nn)$. Without loss of generality we may assume $\eps_{1} +\eps_{5} \in\Delta(\nn)$. Then the relation $\alpha_1+\alpha_2= \beta_{2}+ \eps_{ 1} +\eps_{5}$ is shorter than \refeq{eqConesso(2n)}. Contradiction. Suppose at least one of the roots $\eps_{1}+\eps_{5}$ and $\eps_{2}+\eps_{5}$ belongs to $\Delta(\kk)$. Without loss of generality we may assume $\eps_{1}+\eps_{5}\in\Delta(\kk)$. Then $\eps_{3} - \eps_{5}= \beta_1 -( \eps_{1} + \eps_{5}) \in \Delta(\nn)$ and the relation $\gamma= \eps_{3} - \eps_{5}+\delta$ is shorter than \refeq{eqConesso(2n)}. The latter relation is non-trivial since the right-hand side is a positive linear combination of roots of $\Delta(\nn)$. Contradiction.

So far we proved that $\eps_{1}+\eps_{5},\eps_{2} +\eps_{5}$ do not belong to $\Delta(\lL)$. If $\eps_{1} + \eps_{5}$ were a $\bb\cap\kk$-singular weight, we could replace $\alpha_1+\alpha_2$ by $\eps_{1} + \eps_{5}$ and remove $ \eps_{ 2} + \eps_{3}$ on the right-hand side of \refeq{eqConesso(2n)}, shortening the initial relation. Similarly, we reason that $\eps_{2}+\eps_{5}$ is not a $\bb\cap\kk$-singular weight. In order for $\eps_{ 1}+\eps_{5}$ not to be $\bb\cap\kk$-singular, there must exist an index $k$ and a choice of sign for which one of $\pmf{k}\eps_{k} -\eps_{ 1}$ and $\pmf{k}\eps_{k}-\eps_{5}$ is a positive root of $\kk$. Similarly, there exists an index $l$ and a choice of sign for which one of $\pmf{l}\eps_{l} -\eps_{ 2}$ and $\pmf{l}\eps_{l}-\eps_{5}$ is a positive root of $\kk$. As $\alpha_1$ and $ \alpha_2 $ are $\bb\cap\kk$-singular, a short consideration shows that the only possibility is $\pmf{k}\eps_k=-\eps_2$ and $\pmf{l}\eps_l=\eps_{3}$. Therefore $\beta_3 := \beta_2 -(\eps_{3}-\eps_{5})\in \Delta(\nn)$. Finally, we obtain the relation $\alpha_1+\alpha_2=\beta_1+\beta_3$ which is shorter than \refeq{eqConesso(2n)}. Contradiction.

So far, we have proved that $3\neq j_4$. Therefore we can assume without loss of generality that $\beta_2=\eps_2+\eps_4$. We have now established that the relation has the form
\[
\underbrace{\eps_{1}+\eps_{2}}_{\alpha_1}+\dots= \underbrace{\eps_{1}+\eps_{3}}_{\beta_1}+\underbrace{ \eps_{2} +\eps_{4}}_{\beta_2}+\underbrace{\dots}_{\mathrm{zero~allowed}}.
\]
\noindent Case 1.1. $\eps_{3}$ and $\eps_{4}$ both appear with positive coefficients in $\omega$. We claim that $\alpha_2:= \eps_{3}+ \eps_{4} \in \Delta(\gG)\backslash\Delta(\lL)$. Indeed, first, $\alpha_1=(\beta_1-\alpha_2)+\beta_2$ implies that $\alpha_2 \notin\Delta(\kk)$. Second, if $\alpha_2\in\Delta(\nn)$, we could remove $\alpha_1$ on the left-hand side of the relation and substitute $\beta_1+\beta_2$ by $\alpha_2$ to get a relation shorter than the initial one.

We will now prove that $\alpha_2$ is $\bb\cap\kk$-singular.

Assume on the contrary that there exists $\delta\in\Delta^+(\kk)$ such that $\alpha_2+\delta$ is a root. Then $\delta$ is either of the form $\pmf{k}\eps_{k}-\eps_{3}$ or $\pmf{k}\eps_{k}-\eps_{4}$; without loss of generality we may assume that $\delta=\pmf{k}\eps_{k}-\eps_{3}$. The requirement that $\eps_{3}$ and $\eps_{4}$ appear with positive coefficients in $\omega$ implies that there exist $\alpha_3$, $\alpha_4 \in\hw_{\bb\cap\kk}(\gG/\lL)$ such that $\alpha_3=\pmf{j_5}\eps_{j_5} +\eps_{3}$, $\alpha_4= \pmf{j_6} \eps_{j_6}+\eps_{4}$, $\{j_5,j_6\} \cap\{3,4\} = \emptyset$, $1\neq j_5$, and $2\neq j_6$. Furthermore, the preceding assumptions imply that there are at least three distinct roots on the left-hand side of the relation. Since $\alpha_3$ is $\bb\cap\kk$-singular, we have $j_5=k$ and $\delta= \pmf{k} \eps_{k} -\eps_{3} =  \pmf{j_5} \eps_{j_5}-\eps_{ 3}$. Then $\eps_{1}+ (\pmf{j_5} \eps_{j_5}) = \beta_1 + \delta \in \Delta(\nn)$ and therefore $k=j_5\neq 2$. We can now assume without loss of generality that ${j_5}=5$ and the choice of $\pm$ signs is such that $\alpha_3=\eps_5+\eps_3$ and $\delta=\eps_5-\eps_3\in\Delta^+(\kk)$. So far, the assumption that $\alpha_2$ is not $\bb\cap\kk$-singular implies that the relation has the form
\[
\underbrace{\eps_{1}+\eps_{2}}_{\alpha_1}+\underbrace{\eps_{3}+\eps_{5}}_{\alpha_3} +\underbrace{ \eps_{4}+ (\pmf{j_6} \eps_{j_6})}_{\alpha_4} +\underbrace{\dots}_{\mathrm{zero~allowed}} = \underbrace{\eps_{1} + \eps_{3} }_{\beta_1}+\underbrace{ \eps_{2} +\eps_{4}}_{\beta_2}+\underbrace{\dots}_{\mathrm{zero~allowed}}.
\]
We have that $\eps_{5} +\eps_{4}=\alpha_2+\delta\in\Delta(\gG)\backslash\Delta(\lL)$. We claim that $\eps_{5} +\eps_{4}$ is not $\bb\cap\kk$-singular: indeed, otherwise the relation $\alpha_1+\eps_{5} +\eps_{4}= \underbrace{\eps_{1} +\eps_{5}}_{\in\Delta(\nn)} +\beta_2$ would be shorter than the initial one. Therefore there is a root $\delta'\in \Delta^+(\kk)$ such that $\delta'+\eps_{5} +\eps_{4}$ is a root. The $\bb\cap\kk$-singularity of $\alpha_4$ together with $\delta=\eps_5-\eps_3\in\Delta^+(\kk)$ imply that $\delta' = \pmf{j_6}\eps_{j_6}-\eps_{4}$. Therefore $\eps_{2}+(\pmf{j_6} \eps_{j_6}) \in \Delta(\nn)$ and if the weight $\eps_{5}+ (\pmf{j_6}\eps_{j_6})$ is a root, it belongs to $\Delta(\gG) \backslash\Delta(\lL)$. We can write 
\begin{equation}\label{eqTempso(2n)} 
\alpha_1+\eps_{5}+(\pmf{j_6}\eps_{j_6})=\eps_{1}+\eps_{5} + \eps_{2}+ (\pmf{j_6} \eps_{j_6}). 
\end{equation} 
We will arrive at a contradiction for all possible choices of $j_6$. Indeed, if $5\neq j_6$, then $\eps_{5}+(\pmf{j_6}\eps_{j_6})$ is a root. The fact that  $\alpha_2,\alpha_3\in\hw_{\bb\cap\kk}(\gG/\lL)$ together with $\delta, \delta'\in \Delta^+(\kk)$ imply that $\eps_{5}+(\pmf{j_6}\eps_{j_6})$ is $\bb\cap\kk$-singular. Thus \refeq{eqTempso(2n)} is a relation of type (\ref{eqRelation}) which is shorter than the initial one. Contradiction. If $j_6=5$ and the choice of the sign $\pmf{j_6}$ is such that $\alpha_4=\eps_4-\eps_5 $, we get a contradiction as $-\eps_4+\eps_5= \delta' \in \Delta (\kk)$. Finally, if $\eps_{5}=\pmf{j_6}\eps_{j_6}$, then $\delta'':= -\delta+\delta'= \eps_{3}-\eps_{4}\in\Delta(\kk)$. Then depending on whether $\delta''$ is positive or negative we get a contradiction with the $\bb\cap\kk$-singularity of either $\alpha_4$ or $\alpha_3$.

We have now $\alpha_2\in\hw_{\bb\cap\kk}(\gG/\lL)$. Therefore the initial relation is $\alpha_1+\alpha_2=\beta_1+\beta_2$, of type \refeq{eqDcontainsA3'}.

\noindent Case 1.2. One of $\eps_{3}$, $\eps_{4}$ appears with positive coefficient in $\omega$ and the other with non-positive. Without loss of generality we may assume that $\eps_{4}$ appears with positive coefficient in $\omega$ and $\eps_3$ with non-positive. Then there exists a root on the right-hand side, say $\beta_3$, of the form $\pmf{k}\eps_{k}- \eps_{ 3}$. The minimality of the relation implies $\beta_3=\eps_{1} - \eps_{ 3}$. So far the relation is
\[
\underbrace{\eps_{1}+\eps_{2}}_{\alpha_1}+\dots=\underbrace{\eps_{1}+ \eps_{3}}_{\beta_1}+ \underbrace{\eps_{1}- \eps_{3}}_{\beta_3} + \underbrace{\eps_{2}+\eps_{4}}_{\beta_2}+\dots
\]
Now consider $\alpha_2:=\eps_{1}+\eps_{4}$. We claim, as in Case 1.1, that $\alpha_2\in\Delta(\gG)\backslash\Delta(\lL)$. Indeed, first, if we had that $\alpha_2\in\Delta(\nn)$, we could substitute $\beta_1+\beta_2+\beta_3$ by $\alpha_2$ on the right-hand side and remove $\alpha_1$ on the left-hand side to obtain a shorter relation than the initial one. Second, $\alpha_1=((\beta_1-\alpha_2)+\beta_2)+\beta_3$ implies $\alpha_2 \notin\Delta(\kk)$. 

Now, as in Case 1.1, we will show that $\alpha_2$ is $\bb\cap\kk$-singular. Indeed, assume the contrary. The fact that $\eps_{4}$ appears with positive coefficient in $\omega$ implies that on the left-hand side there is a $\bb\cap\kk$-singular weight, say $\alpha_3$, of the form $\alpha_3= \pmf{j_5}\eps_{j_5}+\eps_{ 4} $, where $j_5\neq 2$. 

We claim next that  $j_5\neq 1$. 

Indeed, first, if $\pmf{j_5}\eps_{j_5} =-\eps_{1}$, the relation $\alpha_1+\alpha_3=\beta_2$ is shorter than the initial one. Contradiction. Second, $\pmf{j_5}\eps_{j_5} =\eps_{1}$ contradicts the $\bb\cap\kk$-singularity of $\alpha_2$. Therefore $j_5\neq 1,2$ and we can assume without loss of generality that $j_5=5$ and $\alpha_3 =\eps_5 +\eps_4$. The assumption that $\alpha_2=\eps_{ 1}+\eps_{4}$ is not $\bb\cap\kk$-singular implies that there exists some index $l$ for which at least one of $\gamma:=\pmf{l}\eps_{l} -\eps_{4 }$ and $\delta:= \pmf{l} \eps_{l} - \eps_{1}$ belongs to $\Delta^+(\kk)$. The choice $\delta\in\Delta^+(\kk)$ contradicts the $\bb \cap \kk$-singularity of $\alpha_1$ unless $\delta=\eps_{2}-\eps_{1}$. The latter yields a contradiction as well, as it implies $\alpha_1 \in \Delta( \nn)$. The choice $\gamma \in \Delta^+(\kk)$ together with the $\bb \cap \kk$-singularity of $\alpha_3$ implies $\gamma= \eps_{5} -\eps_{ 4}$. Then $\beta_4 = \beta_2 +\gamma\in\Delta(\nn)$ and the relation $\alpha_1+ \alpha_3 = \beta_1+\beta_3+\beta_4$ is of type (\ref{eqRelation}) and is shorter than the initial one. Contradiction.

So far we have proved that $\alpha_2\in\hw_{\bb\cap\kk}(\gG/\lL)$. Therefore $\alpha_1+\alpha_2=\beta_1+\beta_3+\beta_2$ is the desired relation \refeq{eqDcontainsD4'}.

\noindent Case 1.3.  $\eps_{3}$ and $\eps_{4}$ both appear with non-positive coefficients in $\omega$. As $\eps_{3}$ and $\eps_{4}$ are canceled on the right-hand side without contradicting the minimality of the relation, we need to have $\beta_3:=\eps_{1}-\eps_{3}\in\Delta(\nn)$, $\beta_4: =\eps_{2}-\eps_{4}\in\Delta(\nn)$. Thus we have the desired relation \refeq{eqDcontainsD4''}.

\noindent Case 2. There is no index $i$ such that $\alpha_i=\pmf{j_1}\eps_{j_1}+(\pmf{j_2}\eps_{j_2})$ and both $\pmf{j_1} \eps_{j_1}$ and $\pmf{j_2} \eps_{j_2}$ appear with positive coefficients in $\omega$. As $\omega$ is non-trivial, it has at least one non-zero coordinate. Without loss of generality we may assume this to coordinate to be positive, corresponding to $\eps_1$. In addition, without loss of generality, assume that $\alpha_1 =\eps_{1} +\eps_{2}$. By our current assumption, $\eps_{2}$ appears in $\omega$ with non-positive coefficient. Then some $\alpha_i$, say $\alpha_2$, is of the form $\alpha_2 =-\eps_{2} +(\pmf{j_3} \eps_{j_3}) $. 

\noindent Case 2.1. $j_3\neq 1$. Without loss of generality we can assume that $j_3=3$ and $\alpha_2=-\eps_2+\eps_3$. Then $\beta_1:=\alpha_1+\alpha_2$ is a root and by \refle{leminrelationgeneral} we have the desired relation \refeq{eqDcontainsA2'}.

\noindent Case 2.2. $j_3=j_1$ and $\alpha_2=-\eps_{2}+\eps_{1} $. On the right-hand side, there is a root, say $\beta_1$, of the form $\beta_1=\eps_{1}+(\pmf{j_4}\eps_{j_4})$. A short consideration shows that $j_4\neq 1,2$, and so we assume without loss of generality that $\beta_1:=\eps_{1}+\eps_{4}$. The relation so far has the form
\[
\underbrace{\eps_{1}+\eps_{2}}_{\alpha_1} +\underbrace{(-\eps_{2}+\eps_{1})}_{\alpha_2}+\underbrace{\dots}_{\mathrm{allowed~to~be~zero}} =\underbrace{\eps_{1}+\eps_{4}}_{\beta_1}+\dots.
\]

We will now prove that $\eps_{4}$ appears with positive coefficient in $\omega$. Indeed, assume the contrary. Therefore there exists a root on the left-hand side, say $\alpha_3$, of the form $\alpha_3= \eps_{4}+ (\pmf{j_5}\eps_{j_5})$. By \refle{leminrelationgeneral} we get that $j_5\neq 1,2$, and therefore we can assume without loss of generality that $j_5=5$ and $\alpha_3= \eps_{4} +\eps_{5}$. By the requirement of Case 2, $\eps_5$ appears with a non-positive coefficient in $\omega$, and therefore there exists $\alpha_4=-\eps_5+(\pmf{j_6}\eps_{j_6})$. By \refle{leminrelationgeneral}, $\alpha_4+\alpha_3$ is not a root and therefore $\alpha_{4}= \eps_{4} -\eps_{5}$. Therefore we cannot have a shorter relation than
\begin{equation}\label{eqFakeRelationso(2n)}
\underbrace{\eps_{j_1}+\eps_{2}}_{\alpha_1}+ \underbrace{(-\eps_{2}+\eps_{1})}_{\alpha_2}+\underbrace{ \eps_{3}+ \eps_{4}}_{\alpha_3}+ \underbrace{(-\eps_{4}+\eps_{3})}_{\alpha_4}= 2(\underbrace{\eps_{1}+\eps_{3}}_{\beta_1}).
\end{equation}
We claim that the above expression cannot correspond to a minimal relation. Consider $\delta:= \eps_{1}+\eps_{4}$. First, the possibility $\delta \in \Delta(\kk)$ implies $\beta_1 \in\Delta(\gG)\backslash\Delta(\lL)$. Contradiction. Second, the possibility $\delta\in \Delta(\gG) \backslash \Delta(\lL)$ together with the $\bb \cap \kk$-singularity of $\alpha_1,\alpha_2,\alpha_3$ and $\alpha_4$ imply $\delta \in \hw_{\bb\cap\kk}(\gG/\lL)$. In turn this is contradictory since $\delta + \alpha_4 = \beta_1$ is shorter than (\ref{eqFakeRelationso(2n)}). We conclude $\delta\in\Delta(\nn)$. Since in (\ref{eqFakeRelationso(2n)}), the indices $(1,4)$ are symmetric to $(2,3)$, we conclude that $\delta':= \eps_{2} +\eps_{3} \in\Delta(\nn)$. Finally, $\alpha_1+ \alpha_3 = \delta + \delta'$ is a shorter relation than (\ref{eqFakeRelationso(2n)}). Contradiction.

So far, we have proved that  $\eps_{4}$ appears with a non-positive coefficient in $\omega$. Therefore, on the right-hand side there is a root, say $\beta_2$, of the form $\beta_2= -\eps_{4}+( \pmf{j_5}\eps_{j_5})$. The minimality of the relation implies $\beta_2=\eps_1-\eps_4$. Therefore we have the desired relation $\alpha_1+\alpha_2=\beta_1+\beta_2$ of type \refeq{eqDcontainsD3'}.

\subsubsection{$\gG\simeq\so(2n+1)$}

\noindent Case 1. The relation has a short root on the left-hand side, say $\alpha_1$. Without loss of generality we may assume $\alpha_1=\eps_{1}$. The $\bb\cap\kk$-singularity of $\alpha_1$ implies that $\kk$ has no short roots.

\noindent Case 1.1. $\eps_{1}$ appears with a positive coefficient in $\omega$ and therefore there is a root on the right-hand side, say $\beta_1$, of the form $\beta_1= \eps_{1}+ (\pmf{j_2}\eps_{j_2})$. Without loss of generality we may assume $\beta_1=\eps_1+\eps_2$.

\noindent Case 1.1.1. $\eps_{2}$ appears with a non-positive coefficient in $\omega$. As $\eps_2$ must be canceled out without contradicting the minimality of the relation, one of the roots on the right-hand side, say $\beta_2$, is of the form $\beta_2 = \eps_{1} -\eps_{2}$. It is now clear that we cannot have a relation shorter than (\ref{eqBcontainsB2'''}).

\noindent Case 1.1.2. $\eps_{2}$ appears with a positive coefficient in $\omega$. The weight $\eps_{2}$ is not a root of $\kk$. Therefore on the left-hand side of the relation there exists a root, say $\alpha_2$, in which $\eps_{2}$ appears with a positive coefficient.

\noindent Case 1.1.2.1. $\alpha_2$ is short, i.e. $\alpha_2=\eps_2$. The relation is \refeq{eqBcontainsB2'}.

\noindent Case 1.1.2.2. $\alpha_2=\eps_2+(\pmf{j_3}\eps_{j_3})$ is long. We claim that $j_3\neq 1$. Indeed otherwise we would have $\alpha_2=\eps_2-\eps_1$, then $\alpha_1+\alpha_2$ would be a root, and by Lemma \ref{leminrelationgeneral} the relation would be $\alpha_1+\alpha_2=\beta_1$. This is impossible. Therefore $j_3\neq 1$, and without loss of generality we can assume $\alpha_2=\eps_2+\eps_3$. Consider $\beta_2:=\eps_3$;  we claim that $\beta_2\in\Delta(\nn)$. Indeed, we immediately see that $\beta_2\notin\Delta(\kk)$, as otherwise $\alpha_1$ would not be $\bb\cap\kk$-singular. Second, assume $\beta_2 \in\Delta(\gG) \backslash \Delta(\nn)$. If $\beta_2$ were $\bb\cap\kk$-singular, we could shorten the relation by removing $\beta_1$ and replacing $\alpha_1+\alpha_2$ by $\beta_2$. Therefore there exists a root $\gamma\in\Delta^+(\kk)$ such that $\beta_2+\gamma$ is a root. The $\bb\cap\kk$-singularity of $\alpha_1$ and $\alpha_2$ implies that $\gamma=\eps_2-\eps_3$. Therefore $\eps_2\in\Delta(\gG)\backslash\Delta(\lL)$. 

Now consider the relation $\eps_1+\eps_2=\beta_1$. If $\eps_2$ were not $\bb\cap\kk$-singular, there would be a positive root $\gamma\in\kk$ such that $\eps_2+\gamma$ is a root but $\eps_2+\eps_3+\gamma$ is not a root, which is impossible. Thus we have a minimal relation of length two of the form $\eps_1+\eps_2=\beta_1$. Hence the initial relation $\eps_1+(\eps_2+\eps_3)+\dots=\beta_1+\dots$ is also of length two. Therefore the unknowns on the right-hand side sum up to $\eps_3$, which together with Lemma \ref{lestreakOfRoots} implies that $\eps_3\in\Delta(\nn)$. Contradiction. Therefore the relation is $\eps_1+(\eps_2 +\eps_3) = (\eps_1+ \eps_2)+\eps_3$ of type \refeq{eqBcontainsB3'}.

\noindent Case 1.2. $\eps_{1}$ appears with a non-positive coefficient in $\omega$. Therefore there is a root, say $\alpha_2$, of the form $\alpha_2=-\eps_{1}+\eps_{2}$. Now \refle{leminrelationgeneral} implies $\alpha_1+\alpha_2\in\Delta(\nn)$ and we get the desired relation \refeq{eqBcontainsB2''}.

\noindent Case 2. Among all minimal relations there is no relation with short roots on the left-hand side. 

\noindent Case 2.1. On the right-hand side there is a short root, say $\beta_1$. Without loss of generality we may assume $\beta_1= \eps_{1}$. As the relation is minimal, $\eps_{1}$  appears with a positive coefficient in $\omega$. Therefore we can assume without loss of generality that $\alpha_1$ is of the form $\alpha_1= \eps_{1}+ \eps_{2}$.

\noindent Case 2.1.1. $\eps_{2}$ appears with a non-positive coefficient in $\omega$. Then there is a root on the left-hand side, say $\alpha_2$, of the form $\alpha_2= -\eps_{2}+(\pmf{j_3}\eps_{j_3})$. If $\alpha_2\neq \eps_2+\eps_1$, we can apply Lemma \ref{leminrelationgeneral} to get a shorter relation than the initial one. Therefore $\alpha_2=-\eps_2+\eps_1$ and the relation is $\alpha_1+\alpha_2=2\beta_1$, of type \refeq{eqBcontainsB2''''}.

\noindent Case 2.1.2. $\eps_{2}$ appears with a positive coefficient in $\omega$. Since $ \eps_{2}$ cannot be a root of $\nn$ (that would imply $\alpha_1\in\Delta(\nn)$), we have a root, say $\beta_2\in\Delta(\nn)$, of the form $\beta_2=\eps_{2}+(\pmf{j_3}\eps_{j_3})$. Since $j_3\neq 1$ we can assume without loss of generality that $\beta_2=\eps_2+\eps_3$. 

\noindent Case 2.1.2.1. $\eps_{3}$ appears with a positive coefficient in $\omega$. Therefore there is a root, say $\alpha_2$, of the from $\alpha_2=\eps_3+(\pmf{j_4}\eps_{j_4})$. We claim that $j_4=1$. Assume the contrary. Since $j_4\neq 2$, we can assume further without loss of generality that $\alpha_2=\eps_{3}+\eps_4$. A short consideration of all possibilities shows that $\alpha_1+ (\eps_3+\eps_4)+ \dots= \eps_1 + (\eps_2 +\eps_3)+\dots$ must be of length at least 3. Consider the root $\eps_3$. If it were in $\Delta(\nn)$ we could shorten the relation by removing $\alpha_1$ and replacing $\beta_1+ \beta_2 $ by $\eps_{ 3}$. If  $\eps_3$ were in $\Delta(\kk)$, we would get $\alpha_1\in\Delta(\nn)$, which is impossible. Therefore $\eps_{3}\in\Delta(\gG)\backslash\Delta(\lL)$. In a similar fashion, we conclude that $\eps_1+\eps_3 \in \Delta (\gG) \backslash \Delta(\lL)$. If at one of the two roots $\eps_{3}$ or $\eps_1+\eps_3$ were $\bb\cap\kk$-singular, we would get a minimal relation of length 2 - either $\alpha_1+\eps_{3}=\beta_1+\beta_2$ or $\alpha_1 +\eps_1 +\eps_{3} =2\beta_1+\beta_2$. Contradiction. Therefore both $\eps_{3}$ and $\eps_1+\eps_3$ are not $\bb\cap\kk$-singular. This shows that there exists $\gamma\in\Delta^+(\kk)$ such that $\gamma+\eps_1+\eps_3$ is a root. Since $\eps_2-\eps_1$ is not a root of $\Delta(\kk)$ and $\alpha_2$ is $\bb\cap\kk$-singular, we obtain that $\gamma= \eps_4 -\eps_3$. Consider $\alpha:=\eps_1+\eps_4 = \gamma+ \eps_1 +\eps_3$. One checks that the $\bb\cap\kk$-singularity of $\alpha_1$ and $\alpha_2$ implies that $\alpha$ is also $\bb\cap\kk$-singular. Therefore we can shorten the relation by replacing $\alpha_1+\alpha_2$ by $\alpha$ and removing $\beta_2$ on the right-hand side. Contradiction.

So far we have established that $j_4=1$. We have immediately a relation of length two, either $\alpha_1+\eps_3+\eps_{1}=2\eps_1+(\eps_2+\eps_3)$ or $\alpha_1+\eps_3-\eps_1=\eps_2+\eps_3$, and so the initial relation is also of length two.  As there can be no two roots on either side that sum up to a root (see Lemma \ref{leminrelationgeneral}), one quickly checks that the only possibility for the minimal relation (up to $\Delta (\gG )$-automorphism) is $(\eps_1+\eps_2)+(\eps_1+\eps_3)=2\eps_1+(\eps_2+\eps_3)$, i.e. type \refeq{eqBcontainsB3''}.

\noindent Case 2.1.2.2. $\eps_{3}$ appears with a non-positive coefficient in $\omega$. Therefore on the right-hand side there is a root, say $\beta_3$, of the form $\beta_3=\eps_2-\eps_3$. We have a relation of length two: $2(\eps_1+\eps_2)=2\eps_1+(\eps_2+\eps_3)+(\eps_2-\eps_3)$ of type \refeq{eqBcontainsB3'''}. In view of the already fixed data, one quickly checks that the only possibility for the initial relation to be of length two is to coincide with this relation.

\noindent Case 2.2. There is no short root on either side of the minimal relation. Therefore we can repeat verbatim the proof for the case $\gG\simeq \so(2n)$ to obtain that we have one of the relations described for this case. 
\subsubsection{$\gG\simeq\sP(2n)$}
\smallskip\noindent Case 1. $\alpha_i+\alpha_j\notin \Delta(\gG)$ for all $i,j$.

\smallskip\noindent Case 1.1 One of the roots $\alpha_i$, say $\alpha_1$, is short. Without loss of generality we may assume $\alpha_1= \eps_{1}+ \eps_{2}$. Since $\alpha_1+\alpha_j\notin\Delta(\gG)$ for all $j$, both $\eps_{1}$ and $\eps_2$ appear with a positive coefficient in $\omega$. Therefore on the right side of the relation there are roots, say $\beta_1$ and $\beta_2$, of the from $\beta_{1} =\eps_{1}+(\pmf{j_3}\eps_{j_3})$ and $\beta_2=\eps_{2} +(\pmf{j_4} \eps_{j_4})$. Consider the vector $\gamma:= \pmf{j_3} \eps_{j_3}+ (\pmf{j_4} \eps_{j_4})$. The minimality of the relation implies that $\gamma$ is non-zero, and therefore that $\gamma$ is a root. If $\gamma\in\Delta(\nn)$ we could shorten the relation by removing $\alpha_1$ on the left-hand side and replacing $\beta_1 + \beta_2$ by $\gamma$. If $\gamma\in\Delta(\kk)$ then $\alpha_1 = (\beta_1- \gamma ) +\beta_2\in\Delta(\nn)$, which is impossible. Therefore $\gamma\in\Delta(\gG)\backslash\Delta(\lL)$. 

As the relation is minimal, $\pmf{j_3}\eps_{j_3}$ and $\pmf{j_4}\eps_{j_4}$  appear with a positive coefficient in $\omega$. Therefore $\pmf{j_3}\eps_{j_3}$ (respectively, $\pmf{j_4}\eps_{j_4}$) appears also in some root, say $\alpha_3$ (respectively, $\alpha_4$) on the left-hand side. If there existed a root $\delta\in\Delta^+(\kk)$ for which $\gamma+\delta$ is a root, $\delta$ would have a negative coefficient in front of one of $\pmf{j_3}\eps_{j_3}$ or $\pmf{j_4}\eps_{j_4}$. This would contradict the $\bb\cap\kk$-singularity of either $\alpha_{3}$ or $\alpha_{4}$. Therefore we have a minimal relation $\alpha_1+\gamma=\beta_1+\beta_2$. Depending on whether $j_3=j_4$ and whether $j_4=2$ our relation is of type \refeq{eqCcontainsA3'}, \refeq{eqCcontainsC3'} or \refeq{eqCcontainsC3''}.

\smallskip\noindent Case 1.2 All roots $\alpha_i$ are long. Without loss of generality we may assume $\alpha_1=2\eps_{1}$. Since $\alpha_1+\alpha_j\notin\Delta(\gG)$ for all $j$, the weight  $\eps_{1}$ appears with a positive coefficient in $\omega$. Therefore there is a root on the right-hand side, say $\beta_1$, of the form $\beta_1 =\eps_{1}+(\pmf{j_2}\eps_{j_2})$. Without loss of generality we may assume $\beta_1=\eps_1+\eps_2$. If $\eps_{ 2}$ appeared with a non-positive coefficient in $\omega$, there would be a cancellation in the right-hand side of the relation. This is impossible. Thus $\eps_{2}$ appears on the left-hand side and we have the desired relation \refeq{eqCcontainsC2''}.

\smallskip\noindent Case 2. For some $\alpha_i$, $\alpha_j$, we have that $\alpha_i+\alpha_j=\gamma$ is a root. By \refle{leminrelationgeneral} $\gamma\in\Delta(\nn)$ and we get one of the relations \refeq{eqCcontainsA2'}, \refeq{eqCcontainsC2'},  (\ref{eqCcontainsC2''}), or (\ref{eqCcontainsC2'''}).
\end{proof}
\begin{corollary}\label{corConeImpliesNotFKFT}
Let $\gG$ be classical simple and suppose that $\lL$ does not satisfy the cone condition. Then the following statements hold.
\begin{itemize}
\item If $\gG\simeq \sL(n), \so(n),$ or $\so(2n)$, there exists an $\lL$-strictly infinite weight $\omega$.
\item If $\gG\simeq \sP(2n)$ there exists an $\lL$-infinite weight $\omega$.
\end{itemize}
\end{corollary}
\begin{proof} 
The statement for $\sL(n)$ follows from \cite[Lemma 5.4]{PSZ}, so let $\gG\simeq$ $\so(2n)$, $\so(2n+1)$ or $\sP(2n)$. By \refprop{propabundanceHolds}, we can always pick a minimal relation corresponding to a two-sided weight. By direct observation of all possibilities for minimal relations given in Lemma \ref{leBigTedious} we see that all such relations have a strongly orthogonal decomposition with respect to $\Delta(\nn)$ except when $\gG\simeq\sP(2n)$ and the two-sided weight is given by \refeq{eqCcontainsC3'}. 

Suppose now $\gG\simeq\sP(2n)$ and relation \refeq{eqCcontainsC3'} holds. According to the proof of Lemma \ref{leBigTedious} we can assume the relation has the form
\[
\underbrace{\eps_{1}+\eps_{2}}_{\alpha_1}+ \underbrace{2\eps_{3}}_{\alpha_2} = \underbrace{ \eps_{1}+ \eps_{3}}_{\beta_{1}}+\underbrace{\eps_{2}+ \eps_{3} }_{\beta_{2}}.
\]
Consider the root $2\eps_{1}$. If $2 \eps_{1}$ belonged to $\Delta(\kk)$, we would have the contradictory $\alpha_2=\beta_1- (2 \eps_{1})+\beta_1\in\Delta(\nn)$. Similarly, we get $2\eps_{2} \notin \Delta (\kk)$. If both $2\eps_{2}, 2\eps_{1} \in\Delta(\nn)$, we get the new relation $2\alpha_1=2\eps_{1}+2\eps_{2}$ which corresponds to a two-sided weight (since the relation corresponds to a two-sided weight) and this new relation gives an $\lL$-strictly infinite weight. If one of $2\eps_{2}, 2\eps_{1}$, say $2\eps_{1}$, belongs to $\Delta(\gG/\lL)$, it is also $\bb\cap\kk$-singular (otherwise $\alpha_1$ would fail to be $\bb\cap\kk$-singular as well). Therefore we have a new relation 
\begin{equation}
\label{eqSoProblem}\omega':= 2\eps_{1}+ 2{\eps_{3}}=2\beta_1. 
\end{equation}  

We claim that $\omega'$ is $\lL$-infinite. Indeed, let $\tT$ be the subalgebra generated by $\kk$, $\gG^{\pma \beta_1}$, $\gG^{\pma 2\eps_{1}}$ and $\gG^{\pma 2\eps_{3}}$. Let $\nn ':=\nn\cap\tT$. Since $\tT$ contains the Cartan subalgebra $\hh$, $\nn '$ is a direct sum of root spaces and is therefore generated as a $\kk$-module by $\gG^{ \beta_1}$. Let $\sS_1$ be the simple component of $\kk$ whose roots are linked to $2\eps_{1}$; in case there is no such simple component, set $\sS_1:=\{0\}$. Define similarly $\sS_3$ using $2\eps_3$. Then $\sS_1\cap\sS_3=\{0\}$ as otherwise $\Delta(\nn)$ would contain $-\beta_2$. In addition, each $\sS_i$ must be of type $A$, (otherwise it would have a root $2\eps_{i}$). It follows that $\omega'$ is two-sided with respect to $\tT$, and therefore $\omega'$ is $\lL$-infinite.
\end{proof}

\begin{lemma}\label{leRank2Solvable}
Let $\gG$ be a simple Lie algebra of rank 2 and $\lL$ be a solvable root subalgebra (i.e. $\kk=\hh$) which does not satisfy the cone condition. Then there exists an $\lL$-strictly infinite weight.
\end{lemma}
\begin{proof}
We leave the proof of cases $A_2$, $B_2$ and $C_2$  to the reader. We note that in case of type $B_2$ all relations (\ref{eqBcontainsB2'})-(\ref{eqBcontainsB2''''}) appear; similarly, in case of type $C_2$, all relations (\ref{eqCcontainsA2'})-(\ref{eqCcontainsC2'''}) appear.

Let now $\gG\simeq G_2$, and fix the scalar product in $\Delta(\gG)$ so that the length of the long root is $\sqrt{6}$. The following table exhibits one $\lL$-strictly infinite weight in each possible case for $\Delta(\nn)$.

\noindent \begin{tabular}{r|c|c}
\multicolumn{3}{c}{$\gG\simeq G_2$}\\
\multicolumn{1}{c|}{$\omega$} & \begin{tabular}{c}Scalar products. All non-listed \\ scalar products are zero. 
\end{tabular} 
&
\begin{tabular}{p{2cm}} The roots from the relation generate\end{tabular}
\\\hline
$\alpha_1+\alpha_2=3\beta_1$& 
\begin{tabular}{p{4.5cm}}
$\langle\alpha_1,\alpha_2\rangle=\langle\alpha_1,\beta_1\rangle=\langle\alpha_2,\beta_1\rangle=3$, $\langle\alpha_2,\alpha_2\rangle=\langle\alpha_1,\alpha_1\rangle =6$,  $\langle\beta_1,\beta_1\rangle=2$
\end{tabular}
&
$G_2$
\\\hline
$\alpha_1+\alpha_2=\beta_1$ & 
\begin{tabular}{p{4.5cm}}
$\langle\alpha_1,\alpha_2\rangle=1$,  $\langle\alpha_1,\beta_1\rangle=\langle\alpha_2,\beta_1\rangle=3$, $\langle\alpha_2,\alpha_2\rangle = \langle \alpha_1, \alpha_1 \rangle =2$, $\langle\beta_1,\beta_1\rangle=6$
\end{tabular}
&
$G_2$
\\\hline
$2\alpha_1=3\beta_1+\beta_2$ & 
\begin{tabular}{p{4.5cm}}
$\langle\alpha_1,\beta_1\rangle=\langle\alpha_1,\beta_2\rangle=3$, $\langle\beta_1,\beta_1\rangle =2$, $\langle\alpha_1,\alpha_1\rangle=\langle\beta_2,\beta_2\rangle=6$\\
\end{tabular}
&
$G_2$
\\\hline
$\alpha_1+\alpha_2=\beta_1$ & 
\begin{tabular}{p{4.5cm}}
$\langle\alpha_1,\beta_1\rangle=\langle\alpha_2,\beta_1\rangle=3$, $\langle\alpha_1,\alpha_2\rangle =-3$,  $\langle\alpha_1,\alpha_1\rangle=\langle\alpha_2,\alpha_2\rangle=\langle\beta_2,\beta_2\rangle=6$\\
\end{tabular}
&
$A_2$
\\\hline
$\alpha_1+\alpha_2=\beta_1$ &
\begin{tabular}{p{4.5cm}}
$\langle\alpha_1,\beta_1\rangle=\langle\alpha_2,\beta_1\rangle=1$, $\langle\alpha_1,\alpha_2\rangle =-3$, $\langle\alpha_2,\alpha_2\rangle=\langle\beta_1,\beta_1\rangle=2$, $\langle\alpha_1,\alpha_1\rangle=6$\\
\end{tabular}
&
$G_2$
\\\hline
$\alpha_1+\alpha_2=\beta_1$ &
\begin{tabular}{p{4.5cm}}
$\langle\alpha_1,\beta_1\rangle=3$ $\langle\alpha_2,\beta_1\rangle=-1$, $\langle\alpha_1,\alpha_2\rangle =-3$, $\langle\alpha_1,\alpha_1\rangle=\langle\alpha_2,\alpha_2\rangle=\langle\beta_1,\beta_1\rangle=2$\\
\end{tabular}
&
$G_2$
\\\hline
\end{tabular}

\smallskip
\end{proof}

The statement of the following lemma is general, but we will make use of it only for the exceptional Lie algebras.
\begin{lemma}\label{leCentralizerAndNilradical}
Suppose $\nn \cap C(\kk_{ss})$ is not the nilradical of a parabolic subalgebra in $C(\kk_{ss})$ containing $\hh\cap C(\kk_{ss})$. Then the following hold. 
\begin{itemize}
\item[(a)] The cone condition fails. 
\item[(b)] There exists a relation \refeq{eqRelation} of the form given by \refle{leminrelationgeneral}(a) for which $\alpha_1,\alpha_2$ and $\beta_1$ all lie in $\Delta(C(\kk_{ss}))$.  
\item[(c)] There is a relation \refeq{eqRelation} that is $\lL$-infinite.
\end{itemize}
\end{lemma}
\begin{proof}
(a) Suppose on the contrary the cone condition holds. Then there exists $h\in\hh$ such that $h(\beta)>0$ for all $\beta\in\Delta(\nn)$ and $h(\alpha)\leq 0$ for all $\alpha\in\hw_{\bb\cap\kk}(\gG/\lL) \supset \Delta(C(\kk_{ss})$. The element $h$ defines a parabolic subalgebra $(\hh\cap C(\kk_{ss}))+\bigoplus_{\substack{ \gamma\in\Delta(C(\kk_{ss})) \\ \gamma(h)\geq 0}}\gG^\gamma $ of  $C(\kk_{ss})$ whose nilradical is $\nn\cap C(\kk_{ss})$, contradiction. 

(b) Using similar arguments to (a), we see that the cone condition fails when restricted to $\Delta(C(\kk_{ss}))$, i.e. the cones $\Cone_{\ZZ} (\Delta(C(\kk_{ss}))\cap$ $\hw_{\bb\cap\kk}(\gG/\lL))$ and  $\Cone_{\ZZ}$ $(\Delta(\nn) \cap$ $\Delta (C(\kk_{ss})))$ have non-zero intersection. 

Take now a relation \refeq{eqRelation}. Note that $\Delta(C(\kk_{ss}))$ $\cap \hw_{\bb\cap\kk}$ $(\gG/\lL)$ $= \Delta( C(\kk_{ss}))$ $\backslash\Delta(\nn)$. Therefore when we add $-\beta_i$ to both sides of \refeq{eqRelation} we still get a relation of the type \refeq{eqRelation} or zero; thus we can obtain a relation \refeq{eqRelation} with only one term $\beta_1$ on the right-hand side. If we have more than two terms on the left-hand side, by \refle{lestreakOfRoots} we get that the sum of two $\alpha_i$'s must be a root. If that root is in $\Delta(\nn)$, we have found a relation of type given by \refle{leminrelationgeneral}(a); else we can substitute the two roots with their sum and thus reduce the number of terms on the left-hand side. In this fashion, we can reduce the number of summands on the left-hand side to two, which gives the desired relation.

(c) Let $\alpha_1,\alpha_2, \beta_1$ be the roots obtained in (b) and let $\tT$ be the subalgebra generated by $\kk$, $\gG^{\pma\alpha_1}$, $\gG^{\pma\alpha_2}$ and $\gG^{\pma\beta_1}$. Lemma \ref{leRank2Solvable} implies that there exists an $\lL\cap\tT$-strictly infinite weight in $\tT$, which is the desired $\lL$-infinite weight.
\end{proof}

\subsection{Exceptional Lie algebras $G_2$, $F_4$, $E_6$ and $E_7$}
\subsubsection{Exceptional Lie algebra $G_2$}
If $\kk_{ss}=\{0\}$ the existence of an $\lL$-infinite weight is guaranteed by \refle{leCentralizerAndNilradical}. If $\kk_{ss} \neq \{0\}$ it is a straightforward check that, up to a $\gG$-automorphism, the only root subalgebra $\lL=\kk\crplus\nn$ for which the cone condition fails is given by $\Delta(\kk)=\{\pm\gamma_1\}$, $\Delta(\nn)=\{\gamma_1+3\gamma_2, 2\gamma_1+3\gamma_2\}$, where $\gamma_1, \gamma_2$ are positive simple roots of $G_2$ such that $\gamma_1$ is long. For this subalgebra, $(\gamma_1+2\gamma_2)+(\gamma_1+\gamma_2)=2\gamma_1+3\gamma_2$ is the desired $\lL$-(strictly) infinite weight.

Note that $G_2$ is the only simple Lie algebra of rank 2 which admits a non-solvable and non-reductive root Fernando-Kac subalgebra of infinite type (cf. \cite[Example 2]{PS:GenHarishChandra}).

\subsubsection{Exceptional Lie algebras $F_4$, $E_6$, $E_7$}\label{secExceptional}
For a fixed exceptional Lie algebra $\gG$, \refle{leCentralizerAndNilradical} allows us to assume that $\nn\cap C(\kk_{ss})$ is the nilradical of a parabolic subalgebra of $C(\kk_{ss})$ containing $C(\kk_{ss})\cap \hh$. 
The following two lemmas can be proved using a computer; the algorithm we used is described in the next section.

\begin{lemma}\label{leE6}~
Let $\gG\simeq E_6$ or $E_7$ with a root subalgebra $\lL=\kk\crplus\nn$ for which the cone condition fails. Suppose in addition that $\nn \cap C(\kk_{ss})$ is the nilradical of some parabolic subalgebra in $C(\kk_{ss})$ containing $\hh\cap C(\kk_{ss})$. Then there exists an $\lL$-strictly infinite relation \refeq{eqRelation} of one of the types listed for $\so(2n)$ in \refle{leBigTedious} or of the type

\noindent\begin{tabular}{rcl| c |c}
\multicolumn{3}{c|}{$\omega$} & \begin{tabular}{p{4.5cm}} Scalar products.\\ All non-listed scalar products are zero.\end{tabular}&  \begin{tabular}{p{1,3cm}}The roots from the relation generate\end{tabular}
\\\hline
\multicolumn{5}{c}{$\gG\simeq E_6, E_7$}\\\hline
$\alpha_1+\alpha_2+\alpha_3$&$=$&$\beta_1+\beta_1+\beta_3$ &
\begin{tabular}{c}
$\langle\alpha_1,\beta_2\rangle=\langle\alpha_1,\beta_3\rangle=$ \\
$\langle\alpha_2,\beta_1\rangle=\langle\alpha_2,\beta_3\rangle=$ \\
$\langle\alpha_3,\beta_1\rangle=\langle\alpha_3,\beta_2\rangle=1$ \\
\end{tabular} & $A_5$.\\\hline
\end{tabular}
\medskip

\noindent When $\gG\simeq E_6$, the above relation occurs only when $\Delta(\kk)\simeq A_1+A_1+A_1$.
\end{lemma}

For the next lemma, we need to define a special root subalgebra of $\gG\simeq F_4$. Fix the scalar product of the root system of $F_4$ so that the long roots have length $2$. Let $\kk$ be defined by the requirement that $\Delta(\kk_{ss})$ be of type $A_1+A_1$ where both $A_1$ roots are long (all such $\kk$ are conjugate, \cite{Dynkin:Semisimple}). Then $C(\kk_{ss})_{ss}$ is of type $C_2\simeq B_2$. Let $\gamma_1$ and $\gamma_2$ be the positive long roots of $\kk$ and $\beta_1$ and $\beta_2$ be the positive long roots of $C(\kk_{ss})$. Let $\beta_0$ be the unique short root of $\Delta(C(\kk_{ss}))$ which has positive scalar products with both $\beta_1$ and $\beta_2$. The roots $\beta_1,\beta_2,\gamma_1$ and $\gamma_2$ are linearly independent. Let $\beta_3$ be given by the requirement $\langle\beta_1,\beta_3 \rangle=0$, $\langle \beta_2, \beta_3 \rangle=2$, $\langle \gamma_1, \beta_3\rangle=0$, $\langle\gamma_2,\beta_3 \rangle=2$ and let $\beta_4$ be given by the requirement $\langle\beta_1,\beta_4 \rangle=0$, $\langle\beta_2,\beta_4\rangle=2$, $\langle \gamma_1, \beta_4 \rangle=2$, $\langle\gamma_2,\beta_4 \rangle=0$. Then $\gG^{\beta_3}$ and $\gG^{\beta_4}$ generate two $\kk$-submodules of $\gG$, say $\nn'$ and $\nn''$, each of dimension 2. Define $\nn$ as the linear span of $\nn',\nn''$, $\gG^{\beta_0}$, $\gG^{\beta_1}$ and $\gG^{\beta_2}$. Then $\nn$ is a nilpotent subalgebra of $\gG$, and is a $\kk$-module. Further, $\dim\nn = 2+2+(1+1+1)=7$ and $C(\kk_{ss})\cap\nn$ is the nilradical of a parabolic subalgebra of $C(\kk_{ss})$. Set $\lL_1:=\kk\crplus\nn$.  

\begin{lemma}\label{leF4} Let $\gG\simeq F_4$. Suppose in addition that $\nn \cap C(\kk_{ss})$ is the nilradical of some parabolic subalgebra of $C(\kk_{ss})$ containing $\hh\cap C(\kk_{ss})$.
\begin{itemize}
\item[(a)] If $\lL$ is not conjugate to $\lL_1$, there exists an $\lL$-strictly infinite relation \refeq{eqRelation} from the list of Lemma \ref{leBigTedious}. Moreover, all relations from Lemma \ref{leBigTedious} except (\ref{eqCcontainsC3'}) do appear.
\item[(b)] If $\lL$ is conjugate to $\lL_1$, there exists an $\lL$-(non-strictly) infinite relation \refeq{eqRelation}. This relation comes from an $\lL':=\lL\cap\tT$-strictly infinite relation in $\tT$, where $\tT$ is one of the two semisimple subalgebras of type $C_3+A_1$ generated by  $\kk$, $C(\kk_{ss})$ and the conjugate of either $\nn'\cup{\nn'}^-$ or $\nn''\cup{\nn''}^-$. The $\lL'$-strictly infinite relation in $\tT$ can be chosen to be isomorphic to relation (\ref{eqCcontainsA2'}).
 
\end{itemize}
\end{lemma}


Combining \refle{leCentralizerAndNilradical} with Lemmas \ref{leF4} and \ref{leE6} we get the following.

\begin{corollary}\label{corExceptional}
The failure of the cone condition for a root subalgebra $\lL$ of the exceptional Lie algebras of type $F_4, E_6, E_7$ implies the existence of an $\lL$-infinite weight.
\end{corollary} 

\subsubsection{Computer computations for the exceptional Lie algebras}\label{secComputerProof}
This section sketches the algorithm we used to carry out the computer based proofs in \refsec{secExceptional}.

The algorithm has as input the Cartan matrix of a semisimple Lie algebra $\gG$. For a given value of $\lL$, let $S$ be the set of weights of $\hw_{\bb\cap\kk}(\gG/\lL)$ for which $[\gG^{\alpha}, \nn]\subset\nn$ (see Definition \ref{defStronglyPerp}). The output is the following.
\begin{enumerate}
\item[(i)] A list of all possible (up to an automorphism of $\Delta(\gG)$) sets of roots of subalgebras $\lL=\kk\crplus\nn$, for which $C(\kk_{ss})\cap\nn$ is the nilradical of a parabolic subalgebra of $C(\kk_{ss})$ containing $\hh\cap C(\kk_{ss})$.
\item[(ii)] A sublist of the list in (i) for which the corresponding subalgebras do not satisfy the cone condition but there exists no $\lL$-strictly infinite weight of length less than or equal to $\max\{\# S, \rk\gG\}$.

\textbf{Remark.} This sublist turns out to be empty for $\gG\simeq E_6, E_7$ and contains one entry for $\gG\simeq F_4$. This entry corresponds to subalgebras conjugate to $\lL_1$, where $\lL_1$ is the subalgebra defined in section \ref{secExceptional}.
\item[(iii)] A list complementary within (i) to the sublist (ii).

\arxivVersion{\textbf{Remark.} The actual list of $\lL$-strictly infinite weights we produced is more detailed; it includes information about the simple direct summands of $\kk$ whose roots are linked to the roots participating in the relation. 
}
\end{enumerate}
The algorithm follows. \arxivVersion{The actual tables printed out for $\gG\simeq F_4$, $E_6$ and $E_7$ are included in the appendix.}
\begin{itemize}
\item Enumerate (up to a $\gG$-automorphism) all reductive root subalgebras $\kk$ containing $\hh$, according to the classification in \cite{Dynkin:Semisimple}.
\item Fix $\kk$. Compute the $\kk$-module decomposition of $\gG$. Then $\nn$ is given by a set of $\kk$-submodules of $\gG$. 
\item Compute $\Delta(C(\kk_{ss}))$ (Lemma \ref{leC(kss)}). Compute the group $W'$ of all root system automorphisms of $\Delta(\gG)$ which preserve $\Delta(\bb\cap\kk)$. Note that $W'=W'''\rtimes W''$ is the semidirect product of the Weyl group $W'''$ of $C(\kk_{ss})$ with the group $W''$ of graph automorphisms of $(\Delta(\kk_{ss})\oplus \Delta(C(\kk_{ss}))\cap\Delta(\bb)$ which preserve $\Delta(\kk_{ss})$ and $\Delta(C(\kk_{ss}))$ and extend to automorphisms of $\Delta(\gG)$. \arxivVersion{The tables in Appendix \ref{tableGroups} list the cardinalities of the groups $W''$.}
\item Introduce a total order $\prec$ on the set of all sets of $\kk$-submodules of $\gG$ in an arbitrary fashion.
\item Enumerate all relevant possibilities for $\nn$:
\begin{itemize} 
\item Discard all sets of submodules $\mathcal P$ for which there exists $w\in W'$ with $w\Delta(\mathcal P)\prec \Delta(\mathcal P)$ (act element-wise). 
\item Discard all sets of submodules $\mathcal P$ whose union, intersected with $C(\kk_{ss})$, does not correspond to a nilradical of a parabolic subalgebra of $C(\kk_{ss})$.
\end{itemize}
\item Fix $\nn$.
\item Intersect the two cones $\Cone_\QQ(\Delta(\nn))$ and $\Cone_\QQ(\hw_{\bb\cap\kk}(\gG/\lL))$ (by using the simplex algorithm over $\QQ$ to solve the corresponding linear system of inequalities). If the cones intersect, proceed with the remaining steps.
\item Generate the set of weights $S$.
\item Generate all possible couples $\alpha_1,\alpha_2\in S$ ($\alpha_1=\alpha_2$ is allowed) and compute whether $\alpha_1+\alpha_2$ has a strongly orthogonal decomposition with respect to $\Delta(\nn)$. If no such strongly orthogonal decomposition exists, proceed with all triples, quadruples, \dots, up to  $\max\{\# S, \rk\gG\}$-tuples, until reaching a weight with a strongly orthogonal decomposition with respect to $\Delta(\nn)$. If such a strongly orthogonal decomposition is found, add the found $\lL$-infinite weight and $\Delta(\lL)$ to the list (iii), else add it to the list (ii).

\end{itemize}

\section{Proof of Theorem \ref{thMainResult}}\label{secTheProof}\label{secRootSA}
Before we prove Theorem \ref{thMainResult}, we need to prove the following.
\begin{prop}\label{lePara}
Suppose that $\lL$ satisfies the cone condition. 
Then $C(\kk_{ss})\cap N(C(\kk_{ss})\cap\nn)$ is a parabolic subalgebra of $C(\kk_{ss})$. Equivalently, in view of \refle{leC(kss)}, there exists $h\in\hh$ such that 
\begin{equation}\label{eqPara} 
C(\kk_{ss})\cap N(C(\kk_{ss})\cap\nn)=\mathfrak{q}_h:= \hh_1\oplus\bigoplus_{\substack{\alpha(h)\geq 0\\ \alpha\sperp\Delta(\kk_{ss})}} \gG^{\alpha},
\end{equation} 
where $\hh_1=\{h\in\hh~|~\gamma(h)=0$ for all $ \gamma\in\Delta(\kk)\}$. In addition, $C(\kk_{ss})\cap N(\nn) = C(\kk_{ss})\cap N(C(\kk_{ss}) \cap \nn )$.
\end{prop}
\begin{proof}
Throughout the entire proof we use \refle{leC(kss)}. 

The equality $\Cone_{\ZZ}(\hw_{\bb\cap\kk}(\gG/\lL))\cap\Cone_{\ZZ}(\Delta(\nn))= \{0\}$ implies that there exists $h\in\hh$ for which $\beta(h) >0, \forall \beta\in\Delta(\nn)$ and $\alpha(h)\leq 0, \forall \alpha\in\hw_{\bb\cap\kk}(\gG/\lL)$. Let $\mathfrak{q}_h$ be defined as in \refeq{eqPara}. 


We claim first that $\mathfrak{q}_h\supset \left(C(\kk_{ss})\cap N(C(\kk_{ss})\cap\nn)\right)\supset \left( C(\kk_{ss} ) \cap N(\nn)\right)$. Suppose on the contrary that there exists $x:=\underbrace{g}_{\in\hh_1}$ $+ \sum_{ \alpha \sperp\Delta (\kk_{ss})} a_{\alpha} g^\alpha\in C(\kk_{ss})\cap N(C(\kk_{ss})\cap\nn)$ for which there is a root $\gamma\in \Delta(C(\kk_{ss}))$ such that $\gamma(h)<0$ and $a_\gamma\neq 0$. Then $\hh\subset N(C(\kk_{ss})\cap\nn)$ implies that whenever $a_\alpha\neq 0$ we have $g^\alpha\in N(C(\kk_{ss})\cap\nn)$. In particular $g^\gamma\in N(C(\kk_{ss})\cap\nn)$. As $C(\kk_{ss})$ is reductive, ${-\gamma} \in\Delta (C(\kk_{ss}))$, and $-\gamma(h)>0$ implies $-\gamma\in\Delta(\nn)$. Therefore $\gG^{-\gamma}\in C(\kk_{ss})\cap\nn$ which contradicts the inclusion $\gG^{\gamma}\subset N(C(\kk_{ss})\cap\nn)$.

We claim next that $\mathfrak{q}_h\subset C(\kk_{ss})\cap N(C(\kk_{ss})\cap\nn)$. Fix $\alpha\in \Delta( C( \kk_{ss} ))$ for which $\alpha(h) \geq 0$. If $\beta\in\Delta(C(\kk_{ss})\cap\nn)$ and $(\alpha+\beta)$ is a root, then $(\alpha+\beta)(h)>0$. Therefore $\alpha + \beta \in \Delta(C(\kk_{ss}) \cap\nn)$ as all roots in $\Delta(C(\kk_{ss}))$ are $\bb\cap\kk$-singular. Therefore $g^{\alpha}\in N(C(\kk_{ss})\cap\nn)$. 

So far we have established that $\mathfrak{q}_h= C(\kk_{ss})\cap N(C(\kk_{ss})\cap\nn)$; we are left to prove that $\qq_h \subset C(\kk_{ss}) \cap N(\nn)$. Suppose, on the contrary, that there is $-\alpha\in\Delta(C(\kk_{ss}))$ such that  $-\alpha(h)= 0$ and $\gamma:=-\alpha+\beta\in\Delta(\gG)\backslash\Delta(\nn)$ for some $\beta\in\Delta(\nn)$. Since $-\alpha\sperp\Delta(\kk)$, $-\alpha,\alpha\in\hw_{\bb\cap\kk}(\gG/\lL)$ and we have the relation 
\begin{equation}\label{eqBadRel}
\alpha+\gamma=\beta.
\end{equation}
Clearly $\gamma\notin\Delta(\kk)$. For the already fixed choice of $\alpha$, assume that $\gamma\in\Delta(\gG) \backslash \Delta(\lL)$ is a root maximal with respect to the partial order defined by $\bb\cap \kk$, such that there exists a relation \refeq{eqBadRel} as above. If $\gamma\in \hw_{\bb\cap\kk}(\gG/\lL)$, this would contradict the cone condition; therefore there exists $\delta\in\Delta^+(\kk)$ such that $\delta+\gamma$ is a root. The requirement that $\gamma\in\Delta( C(\kk_{ss}))$ forces $\delta$ to be strongly orthogonal to $\alpha$. Therefore $\delta$ has the same scalar product with $\gamma$ as it does with $\beta$, but at the same time $\delta+\gamma$ is a root and $\delta+\beta$ isn't (due to the maximality of $\delta$). We will prove that these requirements are contradictory. Let the simple component of $\gG$ containing $\alpha$, $\gamma$ and $\beta$ be $\sS$.

\begin{itemize}
\item[Case 1] $\sS$ is of type $A$, $D$, $E$ or $G_2$. The inequality $\langle\delta, \beta\rangle= \langle \delta, \gamma \rangle< 0$ contradicts the maximality of $\gamma$ because if it held, we could add $\delta$ on both sides of \refeq{eqBadRel}. The inequality $\langle\delta, \gamma \rangle \geq 0$ implies $\langle\delta, \gamma \rangle = 0$ (the sum of two roots with positive scalar product is never a root). In turn, this contradicts the condition that $\delta+\gamma$ is a root since in root systems of type $A$, $D$, $E$ and $G_2$, strong orthogonality is equivalent to orthogonality.
\item[Case 2] $\sS$ is of type $C$. Without loss of generality we can assume that \refeq{eqBadRel} is $\underbrace{\eps_{j_1}+ \eps_{j_2}}_{\alpha}$ +$\underbrace{ (-\eps_{j_2} + \eps_{j_3})}_{\gamma}$ $= \underbrace{\eps_{j_1}+  \eps_{j_3}}_{\beta}$, where the indices $j_1, j_2, j_3$ are not assumed to be pairwise different. Then $\delta=-\eps_{j_3}+\eps_l$ contradicts the maximality of $\gamma$ for all possible choices of the indices $j_1, j_2, j_3, l$. Furthermore, $\delta=\eps_{j_2}+ \eps_l$ contradicts $\alpha \in \Delta(C(\kk_{ss}))$ for all possible choices of the indices $j_1, j_2, j_3, l$. Contradiction.
\item[Case 3] $\sS$ is of type $B$. 
\begin{itemize}
\item[Case 3.1] $\alpha$ and $\gamma$ are both short. Without loss of generality \refeq{eqBadRel} becomes $\underbrace{\eps_{1}}_{\alpha}$ $+ \underbrace{\eps_{2}}_{\gamma}$ $= \underbrace{\eps_{1} + \eps_{2}}_{\beta}$. The maximality of $\gamma$ implies $\delta=\eps_{1}-\eps_{2}$ which contradicts $\alpha\in\Delta(C(\kk_{ss}))$.
\item[Case 3.2] $\alpha$ is short and $\gamma$ is long. Without loss of generality \refeq{eqBadRel} becomes $\underbrace{\eps_{1}}_{\alpha} + \underbrace{(-\eps_{1}+\eps_{2})}_{\gamma}$ $= \underbrace{\eps_{2}}_{\beta}$. The maximality of $\gamma$ implies $\delta =\eps_{1} \pm\eps_l$ for some index $l$, which contradicts $\alpha\in\Delta(C(\kk_{ss}))$.
\item[Case 3.3] $\alpha$ is long and $\gamma$ is short. Without loss of generality \refeq{eqBadRel} becomes $\underbrace{\eps_{1}+\eps_{2}}_{\alpha}$ +$\underbrace{(-\eps_{2})}_{\gamma} $ $= \underbrace{ \eps_{1} }_{ \beta} $. Thus $\delta=$ $\eps_{2}$ $+\eps_l$ for some index $l$ and $\alpha \in \Delta( C( \kk_{ ss}))$ implies $\delta=\eps_{2}-(\eps_{1})$. Then $\beta + \beta$ $+\delta$ $=\alpha$ yields a contradiction.
\item[Case 3.4] Both $\alpha$ and $\gamma$ are long. Without loss of generality \refeq{eqBadRel} becomes
$\underbrace{ \eps_{1}- \eps_{2}}_{\alpha}$ $+\underbrace{ \eps_{2} +\eps_{3}}_{\gamma}$ $= \underbrace{\eps_{1} +\eps_{3}}_{\beta}$. The assumption that $\delta$ is short contradicts either $\alpha\in \Delta(C(\kk_{ss}))$ or the maximality of the choice of $\gamma$. The fact that all roots participating in \refeq{eqBadRel} together with the root $\delta$ are long is contradictory. Indeed, otherwise we could use the exact same data to obtain a relation (\ref{eqBadRel}) in type $D$.
\end{itemize}
\item[Case 4] $\sS$ is of type $F_4$. Suppose on the contrary that there exist roots $\alpha,\beta,\gamma,\delta$ for which \refeq{eqBadRel} holds and $\delta\sperp\alpha$, $\delta\sperp\beta$, $\delta\perp\gamma$. The same conditions would continue to hold in the root subsystem $\Delta '\supset \beta$ generated by $\alpha, \gamma, \delta $. Since $\Delta'$ is of rank 3, setting $\Delta(\nn'):=\{\beta\}$, $\Delta(\kk_{ss}') =\{\pma\delta\}$ we get data whose existence we proved impossible in the preceding cases. Contradiction.
\end{itemize}
\end{proof}

\begin{proof}[~of \refth{thMainResult}]
First, suppose $\Cone_{\ZZ} (\hw_{\bb\cap\kk}(\gG/\lL))\cap \Cone_{\ZZ}(\Delta(\nn))\neq \{0\}$. Then $\lL$ is a Fernando-Kac subalgebra of infinite type by Lemma \ref{leNotFK} and \refsec{secBigTedious}.

Second, suppose $\Cone_{\ZZ} (\hw_{\bb\cap\kk}(\gG/\lL))\cap \Cone_{\ZZ}(\Delta(\nn))= \{0\}$ but $[C(\kk_{ss})\cap N(\nn)]$ has a Levi subalgebra that has a simple component of type $B$, $D$, or $E$. Let $h\in\hh$ be such that $\gamma(h)>0$ for all $\gamma \in \Cone_{\ZZ}(\Delta(\nn))$ and $\gamma(h)\leq 0$ for all $\gamma\in\hw_{\bb\cap\kk}(\gG/\lL)$. According to Proposition \ref{lePara}, $[C(\kk_{ss}) \cap N(\nn)] =[C(\kk_{ss})\cap N(C(\kk_{ss})\cap\nn)]=\qq_h$, where $\qq_h$ is defined as in \refle{lePara}. Assume on the contrary that there exists an irreducible $(\gG,\lL)$-module with $\gG[M]=\lL$. Pick an arbitrary $\bb\cap\lL$-singular vector $v$ and consider the $\qq_h\cap C(\kk_{ss})_{ss}$-module $N$ generated by $v$. We have that $N$ is a strict $(\qq_h\cap C(\kk_{ss})_{ss}, \hh\cap C(\kk_{ss})_{ss})$-module (``torsion-free'' according to the terminology of \cite{Fernando1}). Then, according to \cite[Theorem 5.2]{Fernando1}, it cannot have finite-dimensional $\hh\cap C(\kk_{ss})_{ss}$-weight spaces. In particular there are infinitely many $u_1,\dots\in U(C(\kk_{ss}))$ such that $u_1\cdot v,\dots$ are linearly independent and of same $\hh$-weight. Then $u_i\in U(C(\kk_{ss}))$ implies $u_1\cdot v,\dots$ are all $\bb\cap\kk$-singular, which contradicts the fact that $M$ is of finite type over $\kk$ (\cite[Theorem 3.1]{PSZ}). 


Third, suppose $\Cone_{\ZZ}(\hw_{\bb\cap\kk}(\gG/\lL))\cap \Cone_{\ZZ}(\Delta(\nn))=\{0\}$ and $C(\kk_{ss})\cap N(\nn)$ has simple Levi components of type $A$ and $C$ only.  We will prove that $\lL$ is Fernando-Kac subalgebra of finite type by the construction \cite[Theorem 4.3]{PSZ}. Since the cones do not intersect, there exists a hyperplane in $\hh^*$ given by an element $h\in\hh$ such that $\Delta(\nn))$ lies in the $h$-strictly positive half-space, and $\Cone_{\ZZ}(\hw_{\bb\cap\kk}(\gG/\lL))$ lies in the $h$-non-positive half-space. Clearly we can assume $h$ to have rational action on $\hh^*$.

We introduce now a ``small perturbation'' procedure for $h$ to produce an element $h'$  such that $\gamma(h')\neq 0$ for all $\gamma\in\Delta(\kk)$. Suppose $\gamma\in \Delta(\bb\cap\kk)$ is a root with $\gamma(h)=0$. Define $g\in\hh$ by the properties $\gamma(g)=1$, $\gamma'(g)=0$ for all $\gamma'\perp\gamma$. Now choose $t$ to be a sufficiently small positive rational number ($t\leq\half \min_{\beta \in \Delta(\gG), \beta(h)\neq 0}{|\beta(h)|}$ serves our purpose). Set $h_{1}:=h- tg$. Then all $h$-positive (respectively $h$-negative) vectors remain $h_1$-positive (respectively $h_1$-negative) vectors. The only roots $\alpha$ whose positivity would be affected by the change are those with $\alpha(h)=0$, $\langle\alpha,\gamma\rangle\neq 0$. By the preceding remarks, $\Delta(\nn)$ lies in the $h_1$-positive half-space. We show next that $\Cone_{\ZZ} (\hw_{\bb\cap\kk} (\gG/\lL))$ remains in the $h_1$-non-positive half-space. Suppose on the contrary we had a vector $\alpha\in \hw_{\bb \cap \kk}( \gG/ \lL)$ that now lies in the $h_1$-positive half space. By the preceding remarks $\alpha(h)=0$. Therefore $\alpha(g)= -\frac{1}{t}\alpha(h_1)<0$ which implies $\langle\alpha,\gamma\rangle <0$ and thus $\alpha+\gamma$ is a root. Contradiction. 

If there is a root of $\kk$ that vanishes on $h_1$, we apply the above procedure again, and so on. The number of roots $\alpha\in\Delta(\kk)$ for which $\alpha(h_1)=0$ is smaller than the corresponding number for $h$. Therefore after finitely many iterations we will obtain an element, call it $h'$, for which $\gamma(h')\neq 0$ for all $\gamma\in\Delta(\kk)$ and
\begin{equation}\label{eqTheCondition2}
\begin{array}{c}
\alpha(h')=0 \mathrm{~for~all~} \alpha \mathrm{~for~which~}\gG^\alpha\in C(\kk_{ss}).
\end{array}
\end{equation}


Now define 
\[
\pp:=\bigoplus_{\alpha(h')\geq 0}\gG^{\alpha}, \pp_h:=\bigoplus_{\alpha(h)\geq 0}\gG^{\alpha}.
\]
Then $\pp_{red}\subset{(\pp_h)}_{red}=\hh+\qq_h$, where $\qq_h$ is the subalgebra defined in \refle{lePara}. By \refle{lePara}, we get ${\qq_h}=C(\kk_{ss})\cap N(\nn)$ and the latter is direct sum of simple components of type $A$ and $C$ by the centralizer condition. Thus $\pp_{red}$ is a sum of root systems of type $A$ and $C$ (since types $A$ and $C$ contain root subsystems of type $A$ and $C$ only). We can now pick a $(\pp_{ red} , \hh)$-module $L$  for which $\pp_{red}[L]=\hh$ (see \cite{BrittenLemire}, \cite[Sections 8,9]{Mathieu}), and we can extend $L$ to a $\pp$-module by choosing trivial action of the nilradical of $\pp$. The choice of $h'$ allows us to apply \cite[Theorem 4.3]{PSZ} to get a $\gG$-module $M$ for which $\gG[M]$ is the sum of $\kk$ and the maximal $\kk$-stable subspace of $\pp[L]=\hh\crplus\nn_\pp$. The fact that at least one weight of each irreducible direct summand of $\gG/\lL$ (namely, its $\bb\cap\kk$-singular weight) is outside of $\pp[L]$ implies that the maximal $\kk$-stable subspace of $\pp[L]$ is $\nn$. This completes the proof.
\end{proof}

\addcontentsline{toc}{section}{Bibliography}

\bibliographystyle{amsalpha}
\bibliography{./TodorMilevsBibliography}

\begin{appendix}

\arxivVersion{
\section{Note on table generation}
All tables in the appendix are generated by computer. The tables are also available in .html format from the author or from the web server of the ``vector partition'' program.
\section{Reductive root subalgebras of the exceptional Lie algebras}\label{secRootSATables}
The reductive root subalgebras of the exceptional Lie algebras are described and tabulated in \cite{Dynkin:Semisimple}. Our tables list in addition the type of $C(\kk_{ss})_{ss}$ and the inclusions between the root subsystems parametrizing the reductive root subalgebras.
\subsection{$F_4$}
All diagrams that consist of short roots are labeled by $'$. For example, $A_2'$ has 6 short roots; in the notation $C_3+A_1$, the root system $A_1$ has long roots.

\begin{longtable}{cccc}\multicolumn{4}{c}{$F_4$}\endhead
 \begin{tabular}{p{2cm}}
  $\mathfrak{k}_{ss}$: $B_4$\\
 $C(\mathfrak{k_{ss}})_{ss}$: -\\
 
 $\mathfrak{k}_{ss}$ lies in: \\
 $F_4$, \end{tabular} & \begin{tabular}{p{2cm}}
  $\mathfrak{k}_{ss}$: $D_4$\\
 $C(\mathfrak{k_{ss}})_{ss}$: -\\
 
 $\mathfrak{k}_{ss}$ lies in: \\
 $B_4$, \end{tabular} & \begin{tabular}{p{2cm}}
  $\mathfrak{k}_{ss}$: $C_3$+$A_1$\\
 $C(\mathfrak{k_{ss}})_{ss}$: -\\
 
 $\mathfrak{k}_{ss}$ lies in: \\
 $F_4$, \end{tabular} & \begin{tabular}{p{2cm}}
  $\mathfrak{k}_{ss}$: $A_3$+$A_1'$\\
 $C(\mathfrak{k_{ss}})_{ss}$: -\\
 
 $\mathfrak{k}_{ss}$ lies in: \\
 $F_4$, $B_4$, \end{tabular}\\
 \hline\begin{tabular}{p{2cm}}
  $\mathfrak{k}_{ss}$: $B_2\simeq C_2$+2$A_1$\\
 $C(\mathfrak{k_{ss}})_{ss}$: -\\
 
 $\mathfrak{k}_{ss}$ lies in: \\
 $C_3$+$A_1$, $B_4$, \end{tabular} & \begin{tabular}{p{2cm}}
  $\mathfrak{k}_{ss}$: $A_2'$+$A_2$\\
 $C(\mathfrak{k_{ss}})_{ss}$: -\\
 
 $\mathfrak{k}_{ss}$ lies in: \\
 $F_4$, \end{tabular} & \begin{tabular}{p{2cm}}
  $\mathfrak{k}_{ss}$: 4$A_1$\\
 $C(\mathfrak{k_{ss}})_{ss}$: -\\
 
 $\mathfrak{k}_{ss}$ lies in: \\
 $B_2\simeq C_2$+2$A_1$, \\
 $D_4$, \end{tabular} & \begin{tabular}{p{2cm}}
  $\mathfrak{k}_{ss}$: $C_3$\\
 $C(\mathfrak{k_{ss}})_{ss}$: $A_1$\\
 
 $\mathfrak{k}_{ss}$ lies in: \\
 $C_3$+$A_1$, $F_4$, \end{tabular}\\
\end{longtable}

\subsection{$E_6$}

\begin{longtable}{cccc}\multicolumn{4}{c}{$E_6$}\endhead
\begin{tabular}{p{2cm}}
 $\mathfrak{k}_{ss}$: $A_5$+$A_1$\\
$C(\mathfrak{k_{ss}})_{ss}$: -\\

$\mathfrak{k}_{ss}$ lies in: \\
$E_6$, \end{tabular} & \begin{tabular}{p{2cm}}
 $\mathfrak{k}_{ss}$: 3$A_2$\\
$C(\mathfrak{k_{ss}})_{ss}$: -\\

$\mathfrak{k}_{ss}$ lies in: \\
$E_6$, \end{tabular} & \begin{tabular}{p{2cm}}
 $\mathfrak{k}_{ss}$: $D_5$\\
$C(\mathfrak{k_{ss}})_{ss}$: -\\

$\mathfrak{k}_{ss}$ lies in: \\
$E_6$, \end{tabular} & \begin{tabular}{p{2cm}}
 $\mathfrak{k}_{ss}$: $A_5$\\
$C(\mathfrak{k_{ss}})_{ss}$: $A_1$\\

$\mathfrak{k}_{ss}$ lies in: \\
$A_5$+$A_1$, $E_6$, \end{tabular}\\
\hline\begin{tabular}{p{2cm}}
 $\mathfrak{k}_{ss}$: $A_4$+$A_1$\\
$C(\mathfrak{k_{ss}})_{ss}$: -\\

$\mathfrak{k}_{ss}$ lies in: \\
$E_6$, $A_5$+$A_1$, \end{tabular} & \begin{tabular}{p{2cm}}
 $\mathfrak{k}_{ss}$: $A_3$+2$A_1$\\
$C(\mathfrak{k_{ss}})_{ss}$: -\\

$\mathfrak{k}_{ss}$ lies in: \\
$D_5$, $A_5$+$A_1$, \end{tabular} & \begin{tabular}{p{2cm}}
 $\mathfrak{k}_{ss}$: 2$A_2$+$A_1$\\
$C(\mathfrak{k_{ss}})_{ss}$: -\\

$\mathfrak{k}_{ss}$ lies in: \\
$E_6$, $A_5$+$A_1$, 3$A_2$, \\
\end{tabular} & \begin{tabular}{p{2cm}}
 $\mathfrak{k}_{ss}$: $D_4$\\
$C(\mathfrak{k_{ss}})_{ss}$: -\\

$\mathfrak{k}_{ss}$ lies in: \\
$D_5$, \end{tabular}\\
\hline\begin{tabular}{p{2cm}}
 $\mathfrak{k}_{ss}$: $A_4$\\
$C(\mathfrak{k_{ss}})_{ss}$: $A_1$\\

$\mathfrak{k}_{ss}$ lies in: \\
$A_4$+$A_1$, $D_5$, $A_5$, \\
\end{tabular} & \begin{tabular}{p{2cm}}
 $\mathfrak{k}_{ss}$: $A_3$+$A_1$\\
$C(\mathfrak{k_{ss}})_{ss}$: $A_1$\\

$\mathfrak{k}_{ss}$ lies in: \\
$A_3$+2$A_1$, $D_5$, $A_4$+$A_1$, \\
$A_5$, \end{tabular} & \begin{tabular}{p{2cm}}
 $\mathfrak{k}_{ss}$: 2$A_2$\\
$C(\mathfrak{k_{ss}})_{ss}$: $A_2$\\

$\mathfrak{k}_{ss}$ lies in: \\
2$A_2$+$A_1$, $A_5$, \end{tabular} & \begin{tabular}{p{2cm}}
 $\mathfrak{k}_{ss}$: $A_2$+2$A_1$\\
$C(\mathfrak{k_{ss}})_{ss}$: -\\

$\mathfrak{k}_{ss}$ lies in: \\
2$A_2$+$A_1$, $A_4$+$A_1$, \\
$D_5$, $A_3$+2$A_1$, \end{tabular}\\
\hline\end{longtable}

\subsection{$E_7$}

\begin{longtable}{cccc}\multicolumn{4}{c}{$E_7$}\endhead
  \begin{tabular}{p{2cm}}
   $\mathfrak{k}_{ss}$: $A_7$\\
  $C(\mathfrak{k_{ss}})_{ss}$: -\\
  
  $\mathfrak{k}_{ss}$ lies in: \\
  $E_7$, \end{tabular} & \begin{tabular}{p{2cm}}
   $\mathfrak{k}_{ss}$: $D_6$+$A_1$\\
  $C(\mathfrak{k_{ss}})_{ss}$: -\\
  
  $\mathfrak{k}_{ss}$ lies in: \\
  $E_7$, \end{tabular} & \begin{tabular}{p{2cm}}
   $\mathfrak{k}_{ss}$: $A_5$+$A_2$\\
  $C(\mathfrak{k_{ss}})_{ss}$: -\\
  
  $\mathfrak{k}_{ss}$ lies in: \\
  $E_7$, \end{tabular} & \begin{tabular}{p{2cm}}
   $\mathfrak{k}_{ss}$: $D_4$+3$A_1$\\
  $C(\mathfrak{k_{ss}})_{ss}$: -\\
  
  $\mathfrak{k}_{ss}$ lies in: \\
  $D_6$+$A_1$, \end{tabular}\\
  \hline\begin{tabular}{p{2cm}}
   $\mathfrak{k}_{ss}$: 2$A_3$+$A_1$\\
  $C(\mathfrak{k_{ss}})_{ss}$: -\\
  
  $\mathfrak{k}_{ss}$ lies in: \\
  $D_6$+$A_1$, $E_7$, \end{tabular} & \begin{tabular}{p{2cm}}
   $\mathfrak{k}_{ss}$: 7$A_1$\\
  $C(\mathfrak{k_{ss}})_{ss}$: -\\
  
  $\mathfrak{k}_{ss}$ lies in: \\
  $D_4$+3$A_1$, \end{tabular} & \begin{tabular}{p{2cm}}
   $\mathfrak{k}_{ss}$: $E_6$\\
  $C(\mathfrak{k_{ss}})_{ss}$: -\\
  
  $\mathfrak{k}_{ss}$ lies in: \\
  $E_7$, \end{tabular} & \begin{tabular}{p{2cm}}
   $\mathfrak{k}_{ss}$: $D_6$\\
  $C(\mathfrak{k_{ss}})_{ss}$: $A_1$\\
  
  $\mathfrak{k}_{ss}$ lies in: \\
  $D_6$+$A_1$, $E_7$, \end{tabular}\\
\end{longtable}

\subsection{$E_8$}

\begin{longtable}{cccc}\multicolumn{4}{c}{$E_8$}\endhead
\begin{tabular}{p{2cm}}
 $\mathfrak{k}_{ss}$: $D_8$\\
$C(\mathfrak{k_{ss}})_{ss}$: -\\

$\mathfrak{k}_{ss}$ lies in: \\
$E_8$, \end{tabular} & \begin{tabular}{p{2cm}}
 $\mathfrak{k}_{ss}$: $A_8$\\
$C(\mathfrak{k_{ss}})_{ss}$: -\\

$\mathfrak{k}_{ss}$ lies in: \\
$E_8$, \end{tabular} & \begin{tabular}{p{2cm}}
 $\mathfrak{k}_{ss}$: $E_7$+$A_1$\\
$C(\mathfrak{k_{ss}})_{ss}$: -\\

$\mathfrak{k}_{ss}$ lies in: \\
$E_8$, \end{tabular} & \begin{tabular}{p{2cm}}
 $\mathfrak{k}_{ss}$: $A_7$+$A_1$\\
$C(\mathfrak{k_{ss}})_{ss}$: -\\

$\mathfrak{k}_{ss}$ lies in: \\
$E_7$+$A_1$, $E_8$, \end{tabular}\\
\hline\begin{tabular}{p{2cm}}
 $\mathfrak{k}_{ss}$: $E_6$+$A_2$\\
$C(\mathfrak{k_{ss}})_{ss}$: -\\

$\mathfrak{k}_{ss}$ lies in: \\
$E_8$, \end{tabular} & \begin{tabular}{p{2cm}}
 $\mathfrak{k}_{ss}$: $D_6$+2$A_1$\\
$C(\mathfrak{k_{ss}})_{ss}$: -\\

$\mathfrak{k}_{ss}$ lies in: \\
$E_7$+$A_1$, $D_8$, \end{tabular} & \begin{tabular}{p{2cm}}
 $\mathfrak{k}_{ss}$: $D_5$+$A_3$\\
$C(\mathfrak{k_{ss}})_{ss}$: -\\

$\mathfrak{k}_{ss}$ lies in: \\
$E_8$, $D_8$, \end{tabular} & \begin{tabular}{p{2cm}}
 $\mathfrak{k}_{ss}$: $A_5$+$A_2$+$A_1$\\
$C(\mathfrak{k_{ss}})_{ss}$: -\\

$\mathfrak{k}_{ss}$ lies in: \\
$E_6$+$A_2$, $E_7$+$A_1$, \\
$E_8$, \end{tabular}\\
\hline\begin{tabular}{p{2cm}}
 $\mathfrak{k}_{ss}$: 2$D_4$\\
$C(\mathfrak{k_{ss}})_{ss}$: -\\

$\mathfrak{k}_{ss}$ lies in: \\
$D_8$, \end{tabular} & \begin{tabular}{p{2cm}}
 $\mathfrak{k}_{ss}$: $D_4$+4$A_1$\\
$C(\mathfrak{k_{ss}})_{ss}$: -\\

$\mathfrak{k}_{ss}$ lies in: \\
$D_6$+2$A_1$, 2$D_4$, \end{tabular} & \begin{tabular}{p{2cm}}
 $\mathfrak{k}_{ss}$: 2$A_4$\\
$C(\mathfrak{k_{ss}})_{ss}$: -\\

$\mathfrak{k}_{ss}$ lies in: \\
$E_8$, \end{tabular} & \begin{tabular}{p{2cm}}
 $\mathfrak{k}_{ss}$: 2$A_3$+2$A_1$\\
$C(\mathfrak{k_{ss}})_{ss}$: -\\

$\mathfrak{k}_{ss}$ lies in: \\
$E_7$+$A_1$, $D_6$+2$A_1$, \\
$D_5$+$A_3$, \end{tabular}\\
\hline\end{longtable}

\section{Cardinalities of groups preserving $\Delta(\bb\cap\kk)$}\label{secCardinalitiesIsos}
Define $W'$ to be the group of root system automorphisms of $\Delta(\gG)$ that preserve $\Delta(\bb\cap\kk)$. $W'=W'''\rtimes W''$ is the semidirect sum of the Weyl group $W'''$ of $C(\kk_{ss})$ with the group $W''$ of graph automorphisms of $(\Delta(\kk_{ss})\oplus \Delta(C(\kk_{ss}))\cap\Delta(\bb)$ that preserve $\Delta(\kk_{ss})$ and $\Delta(C(\kk_{ss})$ and extend to root system automorphisms of $\Delta(\gG)$.\label{tableGroups}

\subsection{$F_4$}

\begin{tabular}{ccccc}$\mathfrak{k}_{ss}$& $C(k_{ss})_{ss}$ & $\#W''$ &$\#W'''$&$\#(W'''\rtimes W'')$\\\hline$F_4$ &  & 1 & 1 & 1 \\
$B_4$ &  & 1 & 1 & 1 \\
$D_4$ &  & 6 & 1 & 6 \\
$C_3$+$A_1$ &  & 1 & 1 & 1 \\
$A_3$+$A_1'$ &  & 2 & 1 & 2 \\
$B_2=C_2$+2$A_1$ &  & 2 & 1 & 2 \\
$A_2'$+$A_2$ &  & 2 & 1 & 2 \\
4$A_1$ &  & 24 & 1 & 24 \\
$C_3$ & $A_1$ & 1 & 2 & 2 \\
$B_3$ &  & 1 & 1 & 1 \\
$A_3$ & $A_1'$ & 2 & 2 & 4 \\
$B_2=C_2$+$A_1$ & $A_1$ & 1 & 2 & 2 \\
$A_2'$+$A_1$ &  & 2 & 1 & 2 \\
$A_2$+$A_1'$ &  & 2 & 1 & 2 \\
$A_1'$+2$A_1$ &  & 2 & 1 & 2 \\
3$A_1$ & $A_1$ & 6 & 2 & 12 \\
$B_2=C_2$ & 2$A_1$ & 2 & 4 & 8 \\
$A_2'$ & $A_2$ & 2 & 6 & 12 \\
$A_2$ & $A_2'$ & 2 & 6 & 12 \\
$A_1'$+$A_1$ & $A_1$ & 1 & 2 & 2 \\
2$A_1$ & $B_2=C_2$ & 2 & 8 & 16 \\
$A_1'$ & $A_3$ & 2 & 24 & 48 \\
$A_1$ & $C_3$ & 1 & 48 & 48 \\
- & $F_4$ & 1 & 1152 & 1152 \\
\end{tabular}

\subsection{$E_6$}

\begin{tabular}{ccccc}$\mathfrak{k}_{ss}$& $C(k_{ss})_{ss}$ & $\#W''$ &$\#W'''$&$\#(W'''\rtimes W'')$\\\hline$E_6$ &  & 2 & 1 & 2 \\
$A_5$+$A_1$ &  & 2 & 1 & 2 \\
3$A_2$ &  & 12 & 1 & 12 \\
$D_5$ &  & 2 & 1 & 2 \\
$A_5$ & $A_1$ & 2 & 2 & 4 \\
$A_4$+$A_1$ &  & 2 & 1 & 2 \\
$A_3$+2$A_1$ &  & 4 & 1 & 4 \\
2$A_2$+$A_1$ &  & 4 & 1 & 4 \\
$D_4$ &  & 12 & 1 & 12 \\
$A_4$ & $A_1$ & 2 & 2 & 4 \\
$A_3$+$A_1$ & $A_1$ & 2 & 2 & 4 \\
2$A_2$ & $A_2$ & 4 & 6 & 24 \\
$A_2$+2$A_1$ &  & 4 & 1 & 4 \\
4$A_1$ &  & 48 & 1 & 48 \\
$A_3$ & 2$A_1$ & 4 & 4 & 16 \\
$A_2$+$A_1$ & $A_2$ & 2 & 6 & 12 \\
3$A_1$ & $A_1$ & 12 & 2 & 24 \\
$A_2$ & 2$A_2$ & 4 & 36 & 144 \\
2$A_1$ & $A_3$ & 4 & 24 & 96 \\
$A_1$ & $A_5$ & 2 & 720 & 1440 \\
- & $E_6$ & 2 & 51840 & 103680 \\
\end{tabular}

\subsection{$E_7$}

\begin{tabular}{ccccc}
$\mathfrak{k}_{ss}$& $C(k_{ss})_{ss}$ & $\#W''$ &$\#W'''$&$\#(W'''\rtimes W'')$\\\hline$E_7$ &  & 1 & 1 & 1 \\
$A_7$ &  & 2 & 1 & 2 \\
$D_6$+$A_1$ &  & 1 & 1 & 1 \\
$A_5$+$A_2$ &  & 2 & 1 & 2 \\
$D_4$+3$A_1$ &  & 6 & 1 & 6 \\
2$A_3$+$A_1$ &  & 4 & 1 & 4 \\
7$A_1$ &  & 168 & 1 & 168 \\
$E_6$ &  & 2 & 1 & 2 \\
$D_6$ & $A_1$ & 1 & 2 & 2 \\
$A_6$ &  & 2 & 1 & 2 \\
$D_5$+$A_1$ &  & 2 & 1 & 2 \\
$A_5$+$A_1$ &  & 2 & 1 & 2 \\
$D_4$+2$A_1$ & $A_1$ & 2 & 2 & 4 \\
$A_4$+$A_2$ &  & 2 & 1 & 2 \\
2$A_3$ & $A_1$ & 4 & 2 & 8 \\
$A_3$+$A_2$+$A_1$ &  & 2 & 1 & 2 \\
$A_3$+3$A_1$ &  & 4 & 1 & 4 \\
3$A_2$ &  & 12 & 1 & 12 \\
6$A_1$ & $A_1$ & 24 & 2 & 48 \\
$D_5$ & $A_1$ & 2 & 2 & 4 \\
$A_5$ & $A_1$ & 2 & 2 & 4 \\
$A_5$ & $A_2$ & 2 & 6 & 12 \\
$D_4$+$A_1$ & 2$A_1$ & 2 & 4 & 8 \\
$A_4$+$A_1$ &  & 2 & 1 & 2 \\
$A_3$+$A_2$ & $A_1$ & 2 & 2 & 4 \\
$A_3$+2$A_1$ & $A_1$ & 2 & 2 & 4 \\
2$A_2$+$A_1$ &  & 4 & 1 & 4 \\
$A_2$+3$A_1$ &  & 12 & 1 & 12 \\
5$A_1$ & 2$A_1$ & 8 & 4 & 32 \\
$D_4$ & 3$A_1$ & 6 & 8 & 48 \\
$A_4$ & $A_2$ & 2 & 6 & 12 \\
$A_3$+$A_1$ & 2$A_1$ & 2 & 4 & 8 \\
$A_3$+$A_1$ & $A_3$ & 2 & 24 & 48 \\
2$A_2$ & $A_2$ & 4 & 6 & 24 \\
$A_2$+2$A_1$ & $A_1$ & 4 & 2 & 8 \\
4$A_1$ & 3$A_1$ & 6 & 8 & 48 \\
$A_3$ & $A_3$+$A_1$ & 2 & 48 & 96 \\
$A_2$+$A_1$ & $A_3$ & 2 & 24 & 48 \\
3$A_1$ & 4$A_1$ & 6 & 16 & 96 \\
3$A_1$ & $D_4$ & 6 & 192 & 1152 \\
$A_2$ & $A_5$ & 2 & 720 & 1440 \\
2$A_1$ & $D_4$+$A_1$ & 2 & 384 & 768 \\
$A_1$ & $D_6$ & 1 & 23040 & 23040 \\
- & $E_7$ & 1 & 2903040 & 2903040 \\
\end{tabular}

\subsection{$E_8$}
\begin{longtable}{ccccc}
$\mathfrak{k}_{ss}$& $C(k_{ss})_{ss}$ & $\#W''$ &$\#W'''$&$\#(W'''\rtimes W'')$ \endhead 
$E_8$ &  & 1 & 1 & 1\\
$D_8$ &  & 1 & 1 & 1 \\
$A_8$ &  & 2 & 1 & 2 \\
$E_7$+$A_1$ &  & 1 & 1 & 1 \\
$A_7$+$A_1$ &  & 2 & 1 & 2 \\
$E_6$+$A_2$ &  & 2 & 1 & 2 \\
$D_6$+2$A_1$ &  & 2 & 1 & 2 \\
$D_5$+$A_3$ &  & 2 & 1 & 2 \\
\end{longtable}
\section{$\lL$-infinite weights for the exceptional Lie algebras}\label{secAppendixExceptional}
This section lists all minimal $\lL$-infinite relations and the corresponding two-sided weights under the assumption that $C(\kk_{ss})\cap\nn$ is the nilradical of a parabolic subalgebra of  $C(\kk_{ss})\cap\nn$ containing $C(\kk_{ss})\cap \hh$. 

In the tables to follow, under each root $\alpha_i$ (respectively, $\beta_j$) we write the type of the semisimple component of $\kk$ whose roots are linked to $\alpha_i$ (respectively, $\beta_j$). The ~~~'~~~ sign is used to distinguish different components of $\kk$ that have the same Dynkin type. In type $F_4$, the sign ~~~$'$~~~ stands for a component of $\kk$ whose roots are short. For example, $A_1'+A_1$ represents the direct sum of two $\sL(2)$, one with long and one with short roots, and $A_1$+$A_1$' stands for two long-root $\sL(2)$'s. For example, if a root $\alpha_i$ is linked to $A_1$', and a root $\beta_j$ is linked to $A_1$', then $\alpha_i$ and $\alpha_j$ are linked to the same component of $\Delta(\kk)$; similarly if a root $\alpha_i$ is linked to $A_1$, and a root $\beta_j$ is linked to $A_1$', then the two roots are linked to two different components of $\Delta(\kk)$.
\newpage
\addtolength{\hoffset}{-1.5cm}
\addtolength{\textwidth}{8cm}
\addtolength{\voffset}{-3.3cm}
\addtolength{\textheight}{6.3cm}

\subsection{$F_4$}
\noindent Number of different non-solvable subalgebras up to $\gG$-automorphism such that $\nn\cap C(\kk_{ss})$ is a nilradical of a parabolic subalgebra of $\kk$ containing $C(\kk_{ss})\cap \hh$: 503

\noindent Among them 234 satisfy the cone condition and 269 do not.

\tiny

\begin{longtable}{rcl p{1cm}p{1cm}p{3cm} } \multicolumn{3}{c}{ Relation / linked $\mathfrak{k}$-components} &$\alpha_i$'s, $\beta_i$'s generate & adding $\Delta(\mathfrak{k})$ generates&Non-zero scalar products\endhead
\hline\multicolumn{5}{c}{$\mathfrak{k}$-semisimple type: $A_3$}\\
\hline\hline\begin{tabular}{c}2$\alpha_1$\\\tiny{  }\end{tabular} &\begin{tabular}{c} = \\~ \end{tabular} & \begin{tabular}{cc}$\beta_1$+ & $\beta_2$\\\tiny{ $A_3$ } & \tiny{ $A_3$ }\end{tabular} & $B_2=C_2$ & $B_4$ &  $\langle\alpha_1, \beta_1\rangle=2$, $\langle\alpha_1, \beta_2\rangle=2$, 
\\
\hline\multicolumn{5}{c}{$\mathfrak{k}$-semisimple type: $B_2=C_2$+$A_1$}\\
\hline\hline\begin{tabular}{c}2$\alpha_1$\\\tiny{ $B_2=C_2$ }\end{tabular} &\begin{tabular}{c} = \\~ \end{tabular} & \begin{tabular}{cc}$\beta_1$+ & $\beta_2$\\\tiny{ $B_2=C_2$+$A_1$ } & \tiny{ $B_2=C_2$+$A_1$ }\end{tabular} & $B_2=C_2$ & $F_4$ &  $\langle\alpha_1, \beta_1\rangle=2$, $\langle\alpha_1, \beta_2\rangle=2$, 
\\
\hline\multicolumn{5}{c}{$\mathfrak{k}$-semisimple type: $A_2$+$A_1'$}\\
\hline\hline\begin{tabular}{c}2$\alpha_1$\\\tiny{ $A_1'$ }\end{tabular} &\begin{tabular}{c} = \\~ \end{tabular} & \begin{tabular}{cc}$\beta_1$+ & $\beta_2$\\\tiny{ $A_2$+$A_1'$ } & \tiny{ $A_2$ }\end{tabular} & $B_2=C_2$ & $F_4$ &  $\langle\alpha_1, \beta_1\rangle=2$, $\langle\alpha_1, \beta_2\rangle=2$, 
\\
\hline\multicolumn{5}{c}{$\mathfrak{k}$-semisimple type: $A_1$+$A_1$+$A_1$}\\
\hline\hline\begin{tabular}{c}2$\alpha_1$\\\tiny{ $A_1$ }\end{tabular} &\begin{tabular}{c} = \\~ \end{tabular} & \begin{tabular}{cc}$\beta_1$+ & $\beta_2$\\\tiny{ $A_1$+$A_1$'+$A_1$'' } & \tiny{ $A_1$+$A_1$'+$A_1$'' }\end{tabular} & $B_2=C_2$ & $B_4$ &  $\langle\alpha_1, \beta_1\rangle=2$, $\langle\alpha_1, \beta_2\rangle=2$, 
\\
\end{longtable}

\newpage
\subsubsection[$F_4$ two-sided weights without $\sperp$ decomposition]{$F_4$ two-sided weights without a strongly orthogonal decomposition}
This section lists the only case up to $F_4$-automorphism in $F_4$ for which no two-sided weight with strongly orthogonal decomposition exists. The second table gives one $\lL$-(non-strictly) infinite weight in this case.

\noindent%
 & $C_3$ & $F_4$ &  $\langle\beta_1,\beta_2\rangle=1$, $\langle\alpha_1,\beta_1\rangle=2$, $\langle\alpha_1,\beta_2\rangle=2$, $\langle\alpha_2,\beta_1\rangle=1$, $\langle\alpha_2,\beta_2\rangle=1$, 
\\
\multicolumn{5}{c}{$\varepsilon$-form~relative~to~the~subalgebra~generated~by~$\mathfrak{k}$~and~the~relation}\\
\multicolumn{5}{c}{(2$\varepsilon_{1}$) + (+$\varepsilon_{2}$+$\varepsilon_{3}$)=($\varepsilon_{1}$+$\varepsilon_{3}$) + ($\varepsilon_{1}$+$\varepsilon_{2}$)}\\\hline
\end{tabular}

\medskip ~

\medskip ~

\medskip ~

\medskip ~

\noindent Corresponding $\lL$-(non-strongly) infinite weight.

\medskip ~

\medskip ~

\medskip ~

\noindent\begin{tabular}{rcl p{1cm}p{1cm}p{3cm} } \multicolumn{3}{c}{ Relation / linked $\mathfrak{k}$-components} &$\alpha_i$'s, $\beta_i$'s generate & adding $\mathfrak{k}$ generates&Non-zero scalar products
\\\hline\multicolumn{5}{c}{$\mathfrak{k}$-semisimple type: $A_1$+$A_1$}\\
\hline\hline\begin{tabular}{cc}$\alpha_1$+ & $\alpha_2$\\\tiny{ $A_1$ } & \tiny{  }\end{tabular} &\begin{tabular}{c} = \\~ \end{tabular} & \begin{tabular}{c}$\beta_1$\\\tiny{ $A_1$ }\end{tabular} & $C_3$+$A_1$ & $C_3$+$A_1$ & $\langle\alpha_1,\alpha_2\rangle=-1$, $\langle\alpha_1,\beta_1\rangle=1$, $\langle\alpha_2,\beta_1\rangle=1$, 
\\
\multicolumn{5}{c}{$\varepsilon$-form~relative~to~the~subalgebra~generated~by~$\mathfrak{k}$~and~the~relation}\\
\multicolumn{5}{c}{($\varepsilon_{1}$-$\varepsilon_{2}$) + (+$\varepsilon_{2}$+$\varepsilon_{3}$)=($\varepsilon_{1}$+$\varepsilon_{3}$)}\\\hline
\end{tabular}

\normalsize
\subsection{$E_6$: $\lL$-strictly infinite weights and corresponding relations }
\noindent Number of different non-solvable subalgebras up to $\gG$-automorphism such that $\nn\cap C(\kk_{ss})$ is a nilradical of a parabolic subalgebra of $\kk$ containing $C(\kk_{ss})\cap \hh$: 2044

\noindent Among them 706 satisfy the cone condition and 1338 do not.

\tiny

\begin{longtable}{rcl p{1cm}p{1cm}p{3cm} } \multicolumn{3}{c}{ Relation / linked $\mathfrak{k}$-components} &$\alpha_i$'s, $\beta_i$'s generate & adding $\Delta(\mathfrak{k})$ generates&Non-zero scalar products\\\hline\multicolumn{5}{c}{$\mathfrak{k}$-semisimple type: $A_4$}\\
\hline\hline\begin{tabular}{cc}$\alpha_1$+ & $\alpha_2$\\\tiny{ $A_4$ } & \tiny{  }\end{tabular} &\begin{tabular}{c} = \\~ \end{tabular} & \begin{tabular}{c}$\beta_1$\\\tiny{ $A_4$ }\end{tabular} & $A_2$ & $E_6$ & $\langle\alpha_1, \alpha_2\rangle=-1$,  $\langle\alpha_1, \beta_1\rangle=1$, $\langle\alpha_2, \beta_1\rangle=1$, 
\\
\begin{tabular}{cc}$\alpha_1$+ & $\alpha_2$\\\tiny{ $A_4$ } & \tiny{  }\end{tabular} &\begin{tabular}{c} = \\~ \end{tabular} & \begin{tabular}{cc}$\beta_1$+ & $\beta_2$\\\tiny{ $A_4$ } & \tiny{ $A_4$ }\end{tabular} & $A_3$ & $E_6$ &  $\langle\alpha_1, \beta_1\rangle=1$, $\langle\alpha_1, \beta_2\rangle=1$, $\langle\alpha_2, \beta_1\rangle=1$, $\langle\alpha_2, \beta_2\rangle=1$, 
\\
\hline\multicolumn{5}{c}{$\mathfrak{k}$-semisimple type: $A_3$+$A_1$}\\
\hline\hline\begin{tabular}{cc}$\alpha_1$+ & $\alpha_2$\\\tiny{ $A_3$ } & \tiny{  }\end{tabular} &\begin{tabular}{c} = \\~ \end{tabular} & \begin{tabular}{c}$\beta_1$\\\tiny{ $A_3$ }\end{tabular} & $A_2$ & $A_5$+$A_1$ & $\langle\alpha_1, \alpha_2\rangle=-1$,  $\langle\alpha_1, \beta_1\rangle=1$, $\langle\alpha_2, \beta_1\rangle=1$, 
\\
\begin{tabular}{cc}$\alpha_1$+ & $\alpha_2$\\\tiny{ $A_3$ } & \tiny{ $A_3$ }\end{tabular} &\begin{tabular}{c} = \\~ \end{tabular} & \begin{tabular}{cc}$\beta_1$+ & $\beta_2$\\\tiny{ $A_3$+$A_1$ } & \tiny{ $A_3$+$A_1$ }\end{tabular} & $A_3$ & $E_6$ &  $\langle\alpha_1, \beta_1\rangle=1$, $\langle\alpha_1, \beta_2\rangle=1$, $\langle\alpha_2, \beta_1\rangle=1$, $\langle\alpha_2, \beta_2\rangle=1$, 
\\
\hline\multicolumn{5}{c}{$\mathfrak{k}$-semisimple type: $A_2$+$A_2$}\\
\end{longtable}

\normalsize

\subsection{$E_7$: $\lL$-strictly infinite weights and corresponding relations}
\noindent Number of different non-solvable subalgebras up to $\gG$-automorphism such that $\nn\cap C(\kk_{ss})$ is a nilradical of a parabolic subalgebra of $\kk$ containing $C(\kk_{ss})\cap \hh$: 73834

\noindent Among them 7427 satisfy the cone condition and 66407 do not.
\tiny

\begin{longtable}{rcl p{1cm}p{1cm}p{3cm} } \multicolumn{3}{c}{ Relation / linked $\mathfrak{k}$-components} &$\alpha_i$'s, $\beta_i$'s generate & adding $\Delta(\mathfrak{k})$ generates&Non-zero scalar products\endhead
\hline\multicolumn{5}{c}{$\mathfrak{k}$-semisimple type: $A_3$+$A_3$}\\
\hline\hline\begin{tabular}{c}2$\alpha_1$\\\tiny{  }\end{tabular} &\begin{tabular}{c} = \\~ \end{tabular} & \begin{tabular}{cccc}$\beta_1$+ & $\beta_2$+ & $\beta_3$+ & $\beta_4$\\\tiny{ $A_3$+$A_3$' } & \tiny{ $A_3$+$A_3$' } & \tiny{ $A_3$+$A_3$' } & \tiny{ $A_3$+$A_3$' }\end{tabular} & $D_4$ & $E_7$ &  $\langle\alpha_1, \beta_1\rangle=1$, $\langle\alpha_1, \beta_2\rangle=1$, $\langle\alpha_1, \beta_3\rangle=1$, $\langle\alpha_1, \beta_4\rangle=1$, 
\\
\hline\multicolumn{5}{c}{$\mathfrak{k}$-semisimple type: $D_5$}\\
\hline\hline\begin{tabular}{cc}$\alpha_1$+ & $\alpha_2$\\\tiny{ $D_5$ } & \tiny{  }\end{tabular} &\begin{tabular}{c} = \\~ \end{tabular} & \begin{tabular}{c}$\beta_1$\\\tiny{ $D_5$ }\end{tabular} & $A_2$ & $E_7$ & $\langle\alpha_1, \alpha_2\rangle=-1$,  $\langle\alpha_1, \beta_1\rangle=1$, $\langle\alpha_2, \beta_1\rangle=1$, 
\\
\begin{tabular}{cc}$\alpha_1$+ & $\alpha_2$\\\tiny{ $D_5$ } & \tiny{  }\end{tabular} &\begin{tabular}{c} = \\~ \end{tabular} & \begin{tabular}{cc}$\beta_1$+ & $\beta_2$\\\tiny{ $D_5$ } & \tiny{ $D_5$ }\end{tabular} & $A_3$ & $E_7$ &  $\langle\alpha_1, \beta_1\rangle=1$, $\langle\alpha_1, \beta_2\rangle=1$, $\langle\alpha_2, \beta_1\rangle=1$, $\langle\alpha_2, \beta_2\rangle=1$, 
\\
\hline\multicolumn{5}{c}{$\mathfrak{k}$-semisimple type: $A_5$}\\
\hline\hline\begin{tabular}{cc}$\alpha_1$+ & $\alpha_2$\\\tiny{ $A_5$ } & \tiny{  }\end{tabular} &\begin{tabular}{c} = \\~ \end{tabular} & \begin{tabular}{c}$\beta_1$\\\tiny{ $A_5$ }\end{tabular} & $A_2$ & $A_7$ & $\langle\alpha_1, \alpha_2\rangle=-1$,  $\langle\alpha_1, \beta_1\rangle=1$, $\langle\alpha_2, \beta_1\rangle=1$, 
\\
\end{longtable}


\begin{longtable}{rcl p{1cm}p{1cm}p{3cm} } \multicolumn{3}{c}{ Relation / linked $\mathfrak{k}$-components} &$\alpha_i$'s, $\beta_i$'s generate & adding $\Delta(\mathfrak{k})$ generates&Non-zero scalar products\endhead
\\\hline\multicolumn{5}{c}{$\mathfrak{k}$-semisimple type: $A_2$}\\
\hline\hline\begin{tabular}{cc}$\alpha_1$+ & $\alpha_2$\\\tiny{  } & \tiny{  }\end{tabular} &\begin{tabular}{c} = \\~ \end{tabular} & \begin{tabular}{cc}$\beta_1$+ & $\beta_2$\\\tiny{ $A_2$ } & \tiny{ $A_2$ }\end{tabular} & $A_3$ & $D_5$ &  $\langle\alpha_1, \beta_1\rangle=1$, $\langle\alpha_1, \beta_2\rangle=1$, $\langle\alpha_2, \beta_1\rangle=1$, $\langle\alpha_2, \beta_2\rangle=1$, 
\\
\begin{tabular}{cc}$\alpha_1$+ & $\alpha_2$\\\tiny{ $A_2$ } & \tiny{  }\end{tabular} &\begin{tabular}{c} = \\~ \end{tabular} & \begin{tabular}{c}$\beta_1$\\\tiny{ $A_2$ }\end{tabular} & $A_2$ & $A_4$ & $\langle\alpha_1, \alpha_2\rangle=-1$,  $\langle\alpha_1, \beta_1\rangle=1$, $\langle\alpha_2, \beta_1\rangle=1$, 
\\
\begin{tabular}{cc}$\alpha_1$+ & $\alpha_2$\\\tiny{ $A_2$ } & \tiny{ $A_2$ }\end{tabular} &\begin{tabular}{c} = \\~ \end{tabular} & \begin{tabular}{cc}$\beta_1$+ & $\beta_2$\\\tiny{ $A_2$ } & \tiny{ $A_2$ }\end{tabular} & $A_3$ & $D_5$ &  $\langle\alpha_1, \beta_1\rangle=1$, $\langle\alpha_1, \beta_2\rangle=1$, $\langle\alpha_2, \beta_1\rangle=1$, $\langle\alpha_2, \beta_2\rangle=1$, 
\\
\begin{tabular}{cc}$\alpha_1$+ & $\alpha_2$\\\tiny{  } & \tiny{  }\end{tabular} &\begin{tabular}{c} = \\~ \end{tabular} & \begin{tabular}{ccc}$\beta_1$+ & $\beta_2$+ & $\beta_3$\\\tiny{ $A_2$ } & \tiny{ $A_2$ } & \tiny{ $A_2$ }\end{tabular} & $D_4$ & $E_6$ & $\langle\alpha_1, \alpha_2\rangle=1$,  $\langle\alpha_1, \beta_1\rangle=1$, $\langle\alpha_1, \beta_2\rangle=1$, $\langle\alpha_1, \beta_3\rangle=1$, $\langle\alpha_2, \beta_1\rangle=1$, $\langle\alpha_2, \beta_2\rangle=1$, $\langle\alpha_2, \beta_3\rangle=1$, 
\\
\end{longtable}

\newpage
\begin{longtable}{rcl p{1cm}p{1cm}p{3cm} } \multicolumn{3}{c}{ Relation / linked $\mathfrak{k}$-components} &$\alpha_i$'s, $\beta_i$'s generate & adding $\Delta(\mathfrak{k})$ generates&Non-zero scalar products\endhead
\hline\multicolumn{5}{c}{$\mathfrak{k}$-semisimple type: $A_1$}\\
\hline\hline\begin{tabular}{cc}$\alpha_1$+ & $\alpha_2$\\\tiny{  } & \tiny{  }\end{tabular} &\begin{tabular}{c} = \\~ \end{tabular} & \begin{tabular}{cc}$\beta_1$+ & $\beta_2$\\\tiny{ $A_1$ } & \tiny{ $A_1$ }\end{tabular} & $A_3$ & $D_4$ &  $\langle\alpha_1, \beta_1\rangle=1$, $\langle\alpha_1, \beta_2\rangle=1$, $\langle\alpha_2, \beta_1\rangle=1$, $\langle\alpha_2, \beta_2\rangle=1$, 
\\
\begin{tabular}{cc}$\alpha_1$+ & $\alpha_2$\\\tiny{ $A_1$ } & \tiny{  }\end{tabular} &\begin{tabular}{c} = \\~ \end{tabular} & \begin{tabular}{c}$\beta_1$\\\tiny{ $A_1$ }\end{tabular} & $A_2$ & $A_3$ & $\langle\alpha_1, \alpha_2\rangle=-1$,  $\langle\alpha_1, \beta_1\rangle=1$, $\langle\alpha_2, \beta_1\rangle=1$, 
\\
\begin{tabular}{cc}$\alpha_1$+ & $\alpha_2$\\\tiny{ $A_1$ } & \tiny{ $A_1$ }\end{tabular} &\begin{tabular}{c} = \\~ \end{tabular} & \begin{tabular}{cc}$\beta_1$+ & $\beta_2$\\\tiny{ $A_1$ } & \tiny{ $A_1$ }\end{tabular} & $A_3$ & $D_4$ &  $\langle\alpha_1, \beta_1\rangle=1$, $\langle\alpha_1, \beta_2\rangle=1$, $\langle\alpha_2, \beta_1\rangle=1$, $\langle\alpha_2, \beta_2\rangle=1$, 
\\
\end{longtable}

} 

\end{appendix}

\end{document}